\documentclass[12pt]{amsart}
\usepackage{amssymb, mathdots, a4wide,url,hyperref}
\usepackage{bbm}
\renewcommand{\subset}{\subseteq}

\newtheorem{theorem}{Theorem}[section]
\newtheorem{lemma}[theorem]{Lemma}
\newtheorem{proposition}[theorem]{Proposition}
\newtheorem{remark}[theorem]{Remark}
\newtheorem{conjecture}[theorem]{Conjecture}
\newtheorem*{conjecture-no}{Conjecture}
\newtheorem{definition}[theorem]{Definition}
\newtheorem{corollary}[theorem]{Corollary}

\author{Maxim Gurevich}
\address{Department of Mathematics, National University of Singapore, 10 Lower Kent Ridge Road, Singapore, 119076}
\email{matmg@nus.edu.sg}
\thanks{The author is partially supported by the ISF grant 756/12, and ERC StG grant 637912.}
\date{\today}

\newcommand{\gotM}{\mathfrak{m}}
\newcommand{\gotN}{\mathfrak{n}}

\newcommand{\wid}{\omega}

\DeclareMathOperator{\irr}{Irr}
\DeclareMathOperator{\seg}{Seg}
\DeclareMathOperator{\supp}{supp}

\begin{document}

\title[Decomposition rules]{Decomposition rules for the ring of representations of non-Archimedean $GL_n$}

\begin{abstract}
%Let $\mathcal{R}$ be the Grothendieck ring of complex smooth finite-length representations of the groups $\{GL_n(F)\}_{n=0}^\infty$ ($F$ a fixed $p$-adic field) taken together, with multiplication defined in the sense of parabolic induction. We study the problem of the decomposition of a product of irreducible representations in $\mathcal{R}$.

%We introduce a width invariant for elements of the ring. By showing that the invariant gives an increasing ring filtration, we obtain a necessary condition on irreducible factors of a given product. Irreducible representations of width $1$ form the previously studied class of ladder representations.

%We later focus on the case of a product of two ladder representations, for which we establish that all irreducible factors appear with multiplicity one.

%Finally, we propose a general rule for the composition series of a product of two ladder representations and prove its validity for cases in which the irreducible factors correspond to smooth Schubert varieties.

Let $\mathcal{R}$ be the Grothendieck ring of complex smooth finite-length representations of the sequence of p-adic groups $\{GL_n(F)\}_{n=0}^\infty$, with multiplication defined through parabolic induction. We study the problem of the decomposition of products of irreducible representations in $\mathcal{R}$.

We obtain a necessary condition on irreducible factors of a given product by introducing a width invariant. Width $1$ representations form the previously studied class of ladder representations.

We later focus on the case of a product of two ladder representations, for which we establish that all irreducible factors appear with multiplicity one.

Finally, we propose a general rule for the composition series of a product of two ladder representations and prove its validity for cases in which the irreducible factors correspond to smooth Schubert varieties.

\end{abstract}

\maketitle

\section{Introduction}
Let $F$ be a local non-Archimedean field. Let $\mathcal{R}_n$ be the Grothendieck group associated with the category $\mathfrak{R}_n$ of complex-valued smooth finite-length representations of the group $GL_n(F)$. Given two smooth representations $\pi_i$ of $GL_{n_i}(F)$ ($i=1,2$), $\pi_1\times \pi_2$ is defined as the parabolic induction of $\pi_1\otimes \pi_2$ to $GL_{n_1+n_2}(F)$ from a block upper-triangular maximal Levi subgroup. This product operation equips the group $\mathcal{R} = \oplus_{n\geq 0} \mathcal{R}_n$ with a structure of a commutative ring.

By Zelevinski's work \cite{Zel} the ring is known to be a polynomial ring over $\mathbb{Z}$ in infinitely many variables. One way to observe this is by recalling the Langlands classification, which gives a bijection from the irreducible representations $\irr = \cup_{n\geq 0} \irr(GL_n(F))$ to the so-called standard representations. The collection of standard representations is closed under multiplication and gives a basis to $\mathcal{R}$ as a free abelian group. In particular, the essentially square-integrable (segment representations) elements of $\irr$ freely generate $\mathcal{R}$ as a polynomial ring over $\mathbb{Z}$.

The collection $\irr$ itself gives a different basis to $\mathcal{R}$ as a free abelian group. Our note joins an effort to describe the multiplicative structure of $\mathcal{R}$ in terms of this natural basis. \textit{Given $\pi_1,\pi_2\in \irr$ our goal is to describe the irreducible composition factors of $\pi_1\times\pi_2$ and to compute their multiplicities.}

While our interest in the decomposition problem arose in the setting of $p$-adic groups, this question has a surprisingly omnipresent nature throughout different settings defined by Lie data of type $A$.

First, it is long known \cite{borel-equiv} that the subcategory of $\mathfrak{R}_n$ consisting of representations that admit a Iwahori-invariant vector is equivalent to the category of finite-dimensional modules over the affine Hecke algebra $H(n,F)$ attached to $GL_n(F)$. These equivalences are compatible for our needs, in the sense that they transforms parabolic induction into module induction from a sub-algebra $H(n_1,F)\otimes H(n_2,F)$ to $H(n_1+n_2,F)$. Moreover, $\mathfrak{R}_n$ admits the Bernstein decomposition into components. Each component was shown, in \cite{bk-ss, bk-book} by type theory or in \cite{heier-cat} by other means, again to be equivalent to a category of modules over affine Hecke algebras. See, for example, \cite[Section 3.1]{me-restriction} for a more detailed description. The resulting picture is that our problem may be studied on the level of affine Hecke algebras.

Furthermore, one can take $D$ as a local non-Archimedean division algebra, and consider the smooth representation theory of the groups $GL_m(D)$. Yet again, the Bernstein components of these categories, either by type theory \cite{innform-simple, sech-steve} or by the same result \cite{heier-cat} as in the split case, are known be equivalent to categories of modules over the same affine Hecke algebras. Thus, using these categorical equivalences, our results below remain valid when asked (in the analogous sense, such as discussed, for example, in \cite{LM2}) in the division algebra setting\footnote{A potential alternative approach is to take $D$ in place of $F$ for the full scope of the paper. Yet, as mentioned by an anonymous referee, some claims, such as Lemma \ref{prop-gen}, would necessitate new direct proofs.  }.

Let us also mention the quantum affine Schur-Weyl duality of \cite{cp-duality}, which sends irreducible modules of $H(n,F)$ to irreducible modules of the quantum affine algebra $U_q(\hat{\mathfrak{sl}}_N)$, defined for the affine Lie algebra $\hat{\mathfrak{sl}}_N$ with large enough $N$. Since induction is transformed into tensor product by that duality functor, our problem can also take the form of the decomposition of tensor products of irreducible modules of quantum affine algebras (see \cite[Section 4.1]{lnt} for details).

Our decomposition problem appears in the study of quantum groups from yet another perspective. The ring $\mathcal{R}$ (or rather, its ``principle part") can be viewed as a specialization in $q=1$ of the quantum enveloping algebra $U_q^+(\mathfrak{gl}_\infty)\cong U_q(\mathfrak{n}_\infty)$, for the infinite dimensional Lie algebra $\mathfrak{gl}_\infty = \cup_n \frak{gl}_n$. This type of algebra  is known to be equipped with Lusztig's notion of a canonical basis (or equivalently, Kashiwara's crystal bases). In fact, the basis $\irr$ of the ring $\mathcal{R}$ corresponds to the specialization the dual canonical basis (or upper global crystal basis) of $U_q^+(\mathfrak{gl}_\infty)$, as shown in \cite{lnt} (based on \cite{bz-strings}) and \cite{groj} by two different approaches. The multiplicative properties of the dual canonical basis, such as those studied in \cite{bz-strings} , specialize in $q=1$ back to our questions.

Coming back to $\mathcal{R}$, it is known (see \cite{Hender} for an efficient description) that the transition matrix between the two mentioned above bases, i.e.~standard and irreducible representations, is given by values at $1$ of Kazhdan-Lusztig polynomials for symmetric groups. Thus, given $\pi_1,\pi_2\in \irr$, the irreducible factors of $\pi_1\times \pi_2$ can in principle be determined by computing these polynomials.

However, there is hope that the complexity involved in such a computation may be overcome by applying more direct methods on the problem. For example, as a special case of the problem at hand one can ask whether $\pi_1\times \pi_2$ is already irreducible. A classical result due to Bernstein \cite{bern} gives a positive answer for that question whenever $\pi_1,\pi_2$ are unitarizable. The results of Lapid and M\'{i}nguez \cite{LM2, LM3} further deal with the irreducibility question and supply direct combinatorial criteria in some cases. The thesis of Deng Taiwang \cite{deng} gives rules for decomposition in the case that $\pi_1$ is a segment representation.

One class of irreducible representations which was shown to be susceptible to such methods is that of \textit{ladder} representations, introduced to our setting in \cite{LM}. The accumulated work throughout various settings of representation theory has exhibited remarkable properties of this class and its analogs. In \cite{ram}, Ram identified ladder representations in the context of affine Hecke algebras as precisely those irreducible representations whose Jacquet modules (when translated back to $p$-adic groups terminology) are completely reducible. Using other tools, Kret and Lapid \cite{LapidKret} gave a full description of all Jacquet modules of ladder representations.

Similar phenomena, up to some category equivalences, were previously observed in \cite{naztar, chered-gz} for the analog notion of tame modules of Yangians. Furthermore, \cite{dandan} gives an interpretation of the role of ladder representations for degenerate affine Hecke algebras as analogous to that of finite-dimensional representations in category $\mathcal{O}$. In the context of canonical bases, ladder representations correspond to the quantum minors of \cite{bz-strings}. Back in the setting of $GL_n(F)$, so-called distinction problems were shown in \cite{FLO, my-jaccon, mos} to be especially approachable when dealing with ladder representations.

\subsection*{Results}
The first result of this work presents a new invariant of finite-length representations which attempts to quantify the distance of a given representation from the well-behaved properties of ladder representations.

Given a supercuspidal $\rho\in \irr$, we write $\irr_{\langle \rho \rangle}$ for the set of elements of $\irr$ whose supercuspidal support consists of representations on the supercuspidal line $\{\rho\otimes |\det|^n\}_{n\in\mathbb{Z}}$. Let $\mathcal{R}_{\langle \rho \rangle}$ be the group (which is a subring) generated by $\irr_{\langle \rho \rangle}$ in $\mathcal{R}$. It is enough to study the multiplicative structure of $\mathcal{R}_{\langle \rho \rangle}$
for a fixed $\rho$, since any $\pi\in\irr$ is uniquely decomposed as $\pi= \pi_1\times\cdots\times \pi_k$, where $\pi_i \in \irr_{\rho_i}$, for $\{\rho_i\}$ that belong to disjoint supercuspidal lines.

Recall that the Zelevinski classification\footnote{ Throughout the paper we work with the Langlands classification, but we refer to it in terms of Zelevinski's multisegments. The classifications are dual to each other, as explained for example in \cite{LM2}.} \cite{Zel} of $\irr_{\langle \rho \rangle}$ describes each irreducible representation as a multiset of segments, i.e.~intervals of the form $[a,b],\,a,b\in\mathbb{Z}$. A ladder representation would then be given by a set of the form $\{[a_i,b_i]\}_{i=1}^k$, with $a_1<\ldots<a_k$ and $b_1<\ldots<b_k$.

For any $\pi\in \irr_{\langle \rho \rangle}$, we call the \textit{width} $\wid(\pi)$ of $\pi$ to be the minimal number of ladder sets of segments required to cover the multiset attached to $\pi$. Our first result claims that $\wid$ serves as a length function for an increasing filtration on the ring $\mathcal{R}_{\langle \rho \rangle}$ in the following sense.

\begin{theorem}[Corollary \ref{cor-wid}]\label{main-1}
Let $\rho\in \irr$ be a supercuspidal representation. For $n\in\mathbb{N}$, let $\mathcal{M}_n$ be the subgroup generated in $\mathcal{R}_{\langle \rho \rangle}$ by all $\pi\in \irr_{\langle \rho \rangle}$ with $\wid(\pi)\leq n$. Then,
\[
\mathcal{M}_i\cdot \mathcal{M}_j\subset \mathcal{M}_{i+j},\qquad\forall i,j\in\mathbb{N}.
\]
\end{theorem}

In particular, Theorem \ref{main-1} gives for any $\pi_1,\pi_2\in \irr$ a necessary condition for the occurrence of irreducible representations in the composition series of $\pi_1\times \pi_2$.

The problem of finding a general rule in terms of multisegments for the precise composition series appears to be far from reach (as mentioned, even determining irreducibility is known to be a difficult task). Nevertheless, on the first level of our filtration, namely, when $\pi_1,\pi_2$ are ladder representations, this task was shown to be feasible by the works of Tadi\'{c} \cite{tadic-speh}, Leclerc \cite{lecl} and Ram \cite{ram}. These provide such formulas for a product of ladder representations taken from certain subclasses.

Our second result establishes a general multiplicity-one phenomenon previously observed for each of those subclasses.

\begin{theorem}\label{main-2}
For any ladder representations $\pi_1\in \irr(G_{n_1}),\,\pi_2\in \irr(G_{n_2})$, the irreducible factors of $\pi_1\times\pi_2$ all appear with multiplicity one.

In other words, in the group $\mathcal{R}_{n_1+n_2}$, we have $[\pi_1 \times \pi_2] = [\sigma_1]\, + \,\ldots +\, [\sigma_k]$, for $\sigma_1,\ldots,\sigma_k\in \irr(G_{n_1+n_2})$, with $\sigma_i\ncong \sigma_j$ for $i\neq j$.

\end{theorem}

The proof of Theorem \ref{main-2} relies on a tool we call the \textit{indicator representation} of a representation in $\irr_{\langle \rho \rangle}$. Namely, for any $\sigma\in \irr_{\langle \rho \rangle}$ we construct a specified generic irreducible representation $\sigma_\otimes$ which always occurs as a subquotient of one of the Jacquet modules of $\sigma$. In particular, from exactness of the Jacquet functor we know that the multiplicity of $\sigma$ in a given finite-length representation $\pi$ is bounded by the multiplicity of the indicator $\sigma_\otimes$ in the suitable Jacquet module of $\pi$.

We show that the latter multiplicity can often be computed with a much lower complexity than that which is involved in Kazhdan-Lusztig polynomials computations. Thus, we produce a potentially effective tool for attacking the general decomposition problem.

Another point of view which we adopt is that of permutation indexing of elements of $\irr_{\langle \rho \rangle}$. The symmetric group $S_n$ will give a sort of coordinate systems for irreducible representations constructed on $n$ segments in the Zelevinski classification. With this in mind, we see (Corollary \ref{cor-patt}) that width $k$ irreducible representations are closely related to those parameterized by permutations which avoid the $(k+1)k\ldots 21$ pattern.

In the terms introduced above we formulate a rule for the decomposition of a general product of two ladder representations.
\begin{conjecture-no}[see Conjecture \ref{main-conj} for the precise form]
Given two ladder representations $\pi_1,\pi_2\in \irr_{\langle \rho \rangle}$, the irreducible composition factors of $\pi_1\times\pi_2$ are those representations $\sigma$ which (among the natural candidates) satisfy both two conditions:
\begin{enumerate}
\item The permutation which describes $\sigma$ avoids a $321$ pattern.
\item The indicator $\sigma_\otimes$ appears in a Jacquet module of $\pi_1\times \pi_2$.
\end{enumerate}
\end{conjecture-no}
Both conditions of the conjecture have the advantage of being easily verified for any given Langlands/Zelevinski data of representations.

One particular family of cases of Conjecture \ref{main-conj} can be interpreted in the sense that the ``upper bounds" on the multiplicities of irreducible subquotients supplied by Theorems \ref{main-1} and \ref{main-2} are tight. More precisely, one can construct families of two ladder representations whose product will conjecturally contain all irreducible candidates of the maximal possible width $2$. In fact, these cases of our proposed rule, which are now stated in Proposition \ref{prop-conj}(2), were initially conjectured by Erez Lapid and are shown by the mentioned proposition to imply the full conjecture.

Moreover, Proposition \ref{prop-conj} also shows that Conjecture \ref{main-conj} is equivalent to a family of identities on values of Kazhdan-Lusztig polynomials. These identities were verified by a computer to hold in all low rank cases (more precisely, in all cases when $\pi_1, \pi_2$ are defined by no more than $11$ segments).

In a followup work of the author \cite{me-quantum}, a full proof of Conjecture \ref{main-conj} is given. The proof relies on the reductions obtained here in Proposition \ref{prop-conj}, but requires additional tools from the domain of quantum groups, which extend beyond the natural scope of this paper.

The last part of this work gives a stand-alone proof of Conjecture \ref{main-conj} for certain special cases.
\begin{theorem}[Corollary \ref{cor-smooth}]
Conjecture \ref{main-conj} holds when the Schubert variety associated to the permutation which describes $\sigma$ is smooth.
\end{theorem}

Since we are restricted by the formulation of our conjecture to $321$-pattern avoiding permutations, it is a corollary of the known criterion of \cite{lak-san} that the permutations with a smooth Schubert variety in which we are interested are those also avoiding a $3412$-pattern. In fact, our proof is combinatorial in nature and applies only the pattern avoidance criteria.

\subsection*{Paper structure}
Section \ref{sect-nota} sets up the known generalities about the representation theory of $GL_n(F)$ and the Zelevinski classification of irreducible representations. The tools of Mackey theory for this setting are central for our work.

Section \ref{sect-problem} poses our main problem and covers the basics required for handling it. We also introduce permutation groups and their Kazhdan-Lusztig polynomials into our setting.

Section \ref{sect-wid} deals with our new width invariant, while Section \ref{sect-mulone} contains the proof of the mentioned multiplicity-one theorem.

Section \ref{sect-some-more} proves some properties of indicator representations which will be required in the last section, and might be of stand-alone interest as well.

Section \ref{sect-conj} proposes our conjectural framework for products of two ladder representations, while Section \ref{sect-smooth} proves it for the smooth case.

%\begin{conjecture}\label{conje}
%Let $a_1 < \ldots < a_{2k} < b_1 < \ldots < b_{2k}$ be given integers. Suppose that $\pi\in \irr_{\langle \rho \rangle}$ is the ladder representation associated with the multiset $\{[a_{2i},b_{2i}]\}_{i=1}^k$, $\pi'\in \irr_{\langle \rho \rangle}$ is the ladder representation associated with the multiset $\{[a_{2i-1},b_{2i-1}]\}_{i=1}^k$ and $\lambda$ is the standard representation associated with the union multiset $\{[a_j,b_j]\}_{j=1}^{2k}$. Then, in $\mathcal{R}$, we have

%\[
%[\pi\times\pi'] = \sum_{\begin{array}{cc}\scriptstyle \sigma\in \irr_{\langle \rho \rangle} : \: \scriptstyle  \wid(\sigma)\leq 2  ,\;\\ \scriptstyle\sigma \mbox{ \footnotesize is a subquotient of } \lambda \end{array}} [\sigma]\;.
%\]
%Here $\supp$ denotes the supercuspidal support of a representation.
%\end{conjecture}
%In other words, there is an expected case in which all possible representations of width $2$ should occur in a single product. We expect this case to provide insight towards a formulation of a general rule for the irreducible factors of a product of ladder representations.

%This conjecture is supported by computational evidence for all cases of small number of segments, and was in fact the motivation for introducing the notion of width.

\subsection*{Acknowledgements}
This research was conducted and written as part of the author's work in the Weizmann Institute of Science.

I would like to thank Erez Lapid for introducing the problem to me and for his continuing encouragement all through this project. I am grateful to Bernard Leclerc for sharing many of his valuable insights. Thanks are also due to Arun Ram who enlightened me about many aspects of the problem, and especially for bringing the work of \cite{groj-vaz} to my attention and suggesting the interpretation of the width invariant in terms of Kato modules.

Special thanks to Greg Warrington for sharing some useful Kazhdan-Lusztig polynomials computer scripts on his website, which were used in the verification of Conjecture \ref{main-conj} for low rank cases.

\section{Notation and preliminaries}\label{sect-nota}

\subsection{Generalities}
For a $p$-adic group $G$, let $\mathfrak{R}(G)$ be the category
of smooth complex representations of $G$ of finite length. Denote by $\irr(G)$ the set of equivalence classes
of irreducible objects in $\mathfrak{R}(G)$. Denote by $\mathcal{C}(G)\subset \irr(G)$ the subset of irreducible supercuspidal representations. Let $\mathcal{R}(G)$ be the Grothendieck group of $\mathfrak{R}(G)$. We write $\pi\mapsto [\pi]$ for the canonical map $\mathfrak{R}(G) \to \mathcal{R}(G)$.

Given $\pi\in \mathfrak{R}(G)$, we have $[\pi] = \sum_{\sigma\in \irr(G)} c_\sigma\cdot [\sigma]$. For every $\sigma\in \irr(G)$, let us denote the multiplicity $m(\sigma, \pi):= c_{\sigma}\geq0$. For convenience we will sometimes write $m(\sigma, \pi)=0$ for representations $\pi,\sigma$ of two distinct groups.

Now let $F$ be a fixed $p$-adic field. We write $G_n = GL_n(F)$, for all $n\geq1$, and $G_0$ for the trivial group.

For a given $n$, let $\alpha = (n_1, \ldots, n_r)$ be a composition of $n$. We denote by $M_\alpha$ the subgroup of $G_n$ isomorphic to $G_{n_1} \times \cdots \times G_{n_r}$ consisting of matrices which are diagonal by blocks of size $n_1, \ldots, n_r$ and by $P_\alpha$ the subgroup of $G_n$ generated by $M_\alpha$ and the upper
unitriangular matrices. A standard parabolic subgroup of $G_n$ is a subgroup of the form $P_\alpha$ and its standard Levi factor is $M_\alpha$. We write $\mathbf{r}_\alpha: \mathfrak{R}(G_n)\to \mathfrak{R}(M_\alpha)$ and $\mathbf{i}_\alpha: \mathfrak{R}(M_\alpha)\to \mathfrak{R}(G_n)$ for the normalized Jacquet functor and the parabolic induction functor associated to $P_\alpha$.

Note that naturally $\mathcal{R}(M_\alpha)\cong \mathcal{R}(G_{n_1})\otimes \cdots\otimes \mathcal{R}(G_{n_r})$ and $\irr(M_\alpha)=\irr(G_{n_1})\times\cdots\times \irr(G_{n_r})$.

For $\pi_i\in \mathfrak{R}(G_{n_i})$, $i=1,\ldots,r$, we write
\[
\pi_1\times\cdots\times \pi_r = \mathbf{i}_{(n_1,\ldots,n_r)}(\pi_1\otimes\cdots\otimes \pi_r)\in \mathfrak{R}(G_{n_1+\ldots+n_r}).
\]
The image of a Jacquet functor applied on a representation will often be referred to as a Jacquet module of the representation.

Let us write $\mathcal{R} = \oplus_{m \geq 0} \mathcal{R}(G_m)$. The above product operation $\times$ defines a commutative ring structure on the group $\mathcal{R}$, where the trivial one-dimensional representation of $G_0$ is treated as an identity element.

We also write $\irr = \cup_{m\geq0} \irr(G_m)$ and $\mathcal{C} = \cup_{m\geq1} \mathcal{C}(G_m)$ for the subset of supercuspidal representations.

Given a set $X$, we write $\mathbb{N}(X)$ for the commutative semigroup of maps from $X$ to $\mathbb{N}= \mathbb{Z}_{\geq0}$ with finite support. For $A\in\mathbb{N}(X)$, we write
\[
\underline{A} = \{x\in X \,:\, A(x)>0\}\subset X
\]
for the support of $A$.
%Simplifying notation we will sometimes write $x\in A$ instead of $x\in \underline{A}$.
Given a finite set $S\subset X$, we write $\mathbbm{1}_S\in \mathbb{N}(X)$ for the indicator function of $S$. This gives an embedding $X\to \mathbb{N}(X)$ by $x \mapsto \mathbbm{1}_x$. We will sometimes simply refer to $X$ as a subset of $\mathbb{N}(X)$ by implicitly using this embedding.

For $A\in \mathbb{N}(X)$, we write $\# A= \sum_{x\in X} A(x)$ for the size of $A$.

For $A,B\in \mathbb{N}(X)$ we say that $A\leq B$ if $B-A\in \mathbb{N}(X)$.

\subsection{Langlands classification}
Let us describe the Langlands classification of $\irr$ in terms convenient for our needs.

For any $n$, let $\nu^s= |\det|^s_F,\;s\in \mathbb{C}$ denote the family of one-dimensional representations of $G_n$, where $|\cdot|_F$ is the absolute value of $F$. For $\pi\in \mathfrak{R}(G_n)$, we write $\pi\nu^s := \pi\otimes \nu^s\in \mathfrak{R}(G_n)$.

Given $\rho\in \mathcal{C}(G_n)$ and two integers $a\leq b$, we write $L([a,b]_\rho)\in \irr(G_{n(b-a+1)})$ for the unique irreducible quotient of $\rho\nu^{a}\times \rho\nu^{a+1}\times\cdots\times \rho\nu^b$. It will also be helpful to set $L([a,a-1]_\rho)$ as the trivial representation of $G_0$.

The collection of representations of the form $L([a,b]_\rho)$ is known to exhaust the class of essentially (that is, up to a twist by $\nu^s$) square-integrable irreducible representations.

We also treat the \textit{segment} $\Delta= [a,b]_\rho$ as a formal object defined by the triple $([\rho],a,b)$. We denote by $\seg$ the collection of all segments that are defined by $\rho\in \mathcal{C}$ and integers $a-1\leq b$, up to the equivalence $[a,b]_\rho=[a',b']_{\rho'}$, when $\rho\nu^a\cong \rho'\nu^{a'}$ and $\rho\nu^b \cong \rho'\nu^{b'}$.

A segment $\Delta_1$ is said to precede a segment $\Delta_2$, if $\Delta_1 = [a_1,b_1]_{\rho} ,\;\Delta_2= [a_2,b_2]_{\rho}$ and $a_1\leq a_2-1\leq b_1<b_2$. We will write $\Delta_1 \prec \Delta_2$ in this case and say that the pair $\{\Delta_1,\Delta_2\}$ is linked.

We will write $[a_1,b_1]_{\rho}\subseteq[a_2,b_2]_{\rho}$ when $a_2\leq a_1$ and $b_1\leq b_2$.

The elements of $\mathbb{N}(\seg)$ are called \textit{multisegments}. Langlands classification gives a bijection
\[
L: \mathbb{N}(\seg) \to \irr
\]
that extends the definition of $L$ for a single segment described above.

Given a non-zero multisegment $\gotM$, it is possible to write it as a finite sum of segments in $\seg$. More precisely, we can choose a numbering $\gotM = \Delta_1+\ldots+\Delta_k$, where $\Delta_1,\ldots, \Delta_k\in \seg$, so that $\Delta_j \nprec \Delta_i$ for all $i<j$. We then define the \textit{standard module} $M(\gotM)$ and the \textit{co-standard module} $\widehat{M}(\gotM)$ associated with $\gotM$ to be the representations
\[
M(\gotM) = L(\Delta_k)\times\cdots\times L(\Delta_1),\quad \widehat{M}(\gotM) = L(\Delta_1)\times\cdots\times L(\Delta_k)\;.
\]
The isomorphism classes of $M(\gotM)$ and of $\widehat{M}(\gotM)$ do not depend on the enumeration of segments, as long as the condition above is satisfied. The representation $M(\gotM)$ (respectively, $\widehat{M}(\gotM)$) has a unique irreducible quotient (respectively, sub-representation), which is isomorphic to $L(\gotM)$. We refer to \cite{LM2} for a more thorough discussion of the classification with a similar terminology.

Note, that since $\mathcal{R}$ is a commutative ring, we have $[M(\gotM)] = [\widehat{M}(\gotM)] = [L(\Delta_1)]\times\cdots\times [L(\Delta_k)]$.

When none of the pairs of segments in $\gotM$ are linked the representation $L(\gotM)$ is called \textit{generic}. In that case $M(\gotM)\cong\widehat{M}(\gotM)\cong L(\gotM)$.

\begin{remark}\label{rmk}
We always have $m(L(\gotM_1+\gotM_2), L(\gotM_1)\times L(\gotM_2))=1$, for $\gotM_1,\gotM_2\in \mathbb{N}(\seg)$.
\end{remark}

\subsection{Supercuspidal lines}
For every $\pi\in \irr$ there exist $\rho_1,\ldots,\rho_r \in\mathcal{C}$ for which $\pi$ is a sub-representation of $\rho_1\times\cdots\times \rho_r$. The notion of supercuspidal support can then be defined as
\[
\supp(\pi):= \rho_1+\ldots+\rho_r\in \mathbb{N}(\mathcal{C})\;.
\]
Note that supercuspidal supports can be easily read from the Langlands classification. Namely, for all $\gotM\in \mathbb{N}(\seg)$ and all $\rho\in \mathcal{C}$, $\supp(L(\gotM))(\rho) =
\sum_{\Delta\in \seg\,:\, \rho\in \underline{\supp (\Delta)} }\gotM(\Delta)$, while for a single segment $\Delta=[a,b]_\rho$ we have $\supp(\Delta) = \rho\nu^a + \rho\nu^{a+1}+\cdots+\rho\nu^b$.

Given $\rho\in \mathcal{C}$, we call
\[
\mathbb{Z}_{\langle \rho \rangle}:=\{\rho\nu^a\,:\;a\in \mathbb{Z}\}\subset \mathcal{C}
\]
the \textit{line} of $\rho$.

We write $\irr_{\langle \rho \rangle}\subset \irr$ for the collection of irreducible representations whose supercuspidal support is supported on $\mathbb{Z}_{\langle \rho \rangle}$ (i.e. $\pi\in \irr$ with $\supp(\pi)\in \mathbb{N}(\mathbb{Z}_{\langle \rho \rangle})$). We also write $\seg_{\langle \rho \rangle} = \{[a,b]_\rho\in \seg\,:\;a-1\leq b\}$ and $\mathcal{R}_{\langle \rho \rangle}$ for the ring generated by $\irr_{\langle \rho \rangle}$ in $\mathcal{R}$. It is then straightforward that the restriction of $L$ gives a bijection $\mathbb{N}(\seg_{\langle \rho \rangle})\to \irr_{\langle \rho \rangle}$.

\subsection{Geometric lemma/ Mackey theory/ Shuffle lemma}
The Geometric Lemma of Bernstein-Zelevinski (\cite{BZ1}) is a valuable tool for an analysis of Jacquet functors. It allows for an expression of the Jacquet module of an induced representation as a product of Jacquet modules of the inducing data. This is a natural variant of the Mackey theory (dealing with restriction and induction functors) in the $p$-adic groups setting.

We will review here only a part of the ``semi-simplified" version, i.e. statements on $\mathcal{R}$, of the lemma. For the full ``semi-simplified" statement we refer the reader to \cite[Section 1.2]{LM2}.

Let $\pi_i\in \mathfrak{R}(G_{n_i})$ be given representations, for $i=1,\ldots r$. We can write the decompositions of the Jacquet functor relative to maximal parabolic subgroups as
\[
\sum_{k=0}^{n_i}[\mathbf{r}_{(k,n_i-k)}(\pi_i)] = \sum_{j\in J_i} [\tau_i^j\otimes \delta_i^j]\;,
\]
for all $i$, with $\tau_i^j,\delta_i^j\in\irr$. Let us denote the index set $\mathcal{J}(\pi_1,\ldots,\pi_r) = J_1\times\cdots\times J_r$, and for each $s=(j_1,\ldots, j_r)\in \mathcal{J}(\pi_1,\ldots,\pi_r)$ we write the representations
\[
\tau^s:= \tau_1^{j_1}\times \cdots\times \tau_r^{j_r},\qquad \delta^s:=\delta_1^{j_1}\times\cdots\times \delta_r^{j_r}.
\]
Then, the maximal parabolic subgroups Jacquet modules of the product decompose as follows.
\begin{proposition}[Geometric Lemma]\label{prop-geom}
\[
\sum_{k=0}^n[\mathbf{r}_{(k,n-k)}(\pi_1\times \cdots\times \pi_r)] = \sum_{s\in \mathcal{J}(\pi_1,\ldots,\pi_r)} [\tau^s]\otimes [\delta^s]\;.
\]
\end{proposition}

Mainly in the setting of affine Hecke algebras, the analogue of the proposition above is sometimes referred to as the Shuffle Lemma because of its possible shuffle algebra interpretation. See \cite[Lemma 2.7]{groj-vaz}.

In addition, we will need a certain observation from a slightly more general version of the same lemma.

\begin{definition}
We say that a representation $\kappa\in \irr(G_m)$ is a \textit{Jacquet module component} of $\pi\in \mathfrak{R}(G_n)$ if there is a Levi subgroup $M_\alpha< G_n$ and a representation $\sigma=\sigma_1\otimes\cdots\otimes \sigma_r\in \irr(M_\alpha)$, such that $\kappa\cong \sigma_i$ for some $i$ and $m(\sigma,\mathbf{r}_\alpha(\pi))>0$.

If $i=1$, we will say $\kappa$ is a \textit{leftmost} Jacquet module component.
\end{definition}

\begin{proposition}\label{prop-add}
The collection of Jacquet module components of $\pi_1\times\cdots\times \pi_r\in \mathfrak{R}(G_n)$ is precisely the collection of representations $\sigma\in \irr$ for which $m(\sigma, \kappa_1\times\cdots\times \kappa_k)>0$ holds, for some choice of Jacquet module components $\kappa_i$ of $\pi_i$, for $i=1,\ldots,r$.
\end{proposition}

\section{Main Problem and Tools}\label{sect-problem}
Given representations $\pi^1,\pi^2,\sigma\in \irr$, our goal is to determine the multiplicities
\[
m(\sigma, \pi^1\times\pi^2)\;.
\]

We start with the following simple reduction. If the lines of $\rho_1,\ldots , \rho_k\in\mathcal{C}$ are distinct and $\pi_i \in \irr_{\langle \rho_i \rangle}$, then $\pi_1 \times \cdots \times \pi_k$ is known to be irreducible. Moreover, given $\pi = L(\gotM)\in \irr$, we can uniquely write $\gotM = \sum_{i=1}^k \gotM_i$ with $\gotM_i \in \seg_{\langle \rho_i \rangle}$, such that the lines of  $\rho_1,\ldots,\rho_k\in \mathcal{C}$ are distinct. We thus get a unique decomposition $\pi= \pi_1 \times\cdots\times \pi_k$, with $\pi_i \in \irr_{\langle \rho_i \rangle}$.

It follows that questions of decomposition of induced representations in $\mathcal{R}$ can be solved by analyzing $\mathcal{R}_{\langle \rho \rangle}$, for a single $\rho\in \mathcal{C}$. More precisely, from the above discussion we clearly have
\[
m(\sigma_1\times \cdots\times \sigma_k, (\pi^1_1\times \cdots\times \pi^1_k) \times (\pi^2_1\times \cdots\times \pi^2_k))
= \prod_{i=1}^k m(\sigma_i, \pi^1_i\times\pi^2_i)\;,
\]
where $\pi^1_i,\pi^2_i,\sigma_i\in \irr_{\langle \rho_i \rangle}$ are such that the lines of  $\rho_1,\ldots,\rho_k\in \mathcal{C}$ are distinct.

Out of these considerations, \textit{we choose to fix a single supercuspidal representation $\rho\in \mathcal{C}$ for the remainder of the paper}.

We will naturally identify $\mathbb{Z}_{\langle \rho \rangle}$ with $\mathbb{Z}$. For $\pi\in \irr_{\langle \rho \rangle}$ we will then refer to $\supp(\pi)$ as an element of $\mathbb{N}(\mathbb{Z})$. A non-trivial $\Delta\in \seg_{\langle \rho \rangle}$ can be uniquely written as $\Delta = [a,b]_\rho$. Therefore, we will simply write $\Delta=[a,b]$. We then write $b(\Delta)=a$ and $e(\Delta)=b$.

\subsection{The indicator representation}

Given $\gotM\in \mathbb{N}(\seg_{\langle \rho \rangle})$ with $\pi=L(\gotM)\in \irr(G_n)$, let us denote the collection of integers
\[
B_\gotM = \{b(\Delta)\::\:\Delta\in \underline{\gotM}\}\subset \mathbb{Z}_{\langle \rho \rangle}\cong \mathbb{Z}\;.
\]
We write $B_\gotM = \{b_1,\ldots, b_k\}$ with $b_1<\ldots< b_k$. For each $1\leq i\leq k$, write $S_j = \{\Delta\in \seg_{\langle \rho \rangle}\::\: b(\Delta)=b_j\}$. With these notations we define the following.

\begin{definition}
For $\pi\in \irr(G_n)\cap \irr_{\langle \rho \rangle}$, \textit{the indicator representation} of $\pi$ is
\[
\pi_\otimes = L\left(\gotM\cdot\mathbbm{1}_{S_1}\right)\otimes \cdots\otimes L\left(\gotM\cdot\mathbbm{1}_{S_k}\right)\in \irr(M_{\alpha_\pi})\;,
\]
where $M_{\alpha_\pi}<G_n$ is the corresponding Levi subgroup.
\end{definition}

Since $L(\gotM\cdot\mathbbm{1}_{S_j})$ are generic representations, it is clear that $\widehat{M}(\gotM)\cong \mathbf{i}_{\alpha_\pi}(\pi_\otimes)$. Since $\pi$ is embedded in $\widehat{M}(\gotM)$, by adjunction we always find $\pi_\otimes$ as a quotient of $\mathbf{r}_{\alpha_\pi}(\pi)$. By exactness of the functor $\mathbf{r}_{\alpha_\pi}$, we see that for all $\sigma\in \mathfrak{R}(G_n)$,
\begin{equation}\label{maj}
m(\pi,\sigma) \leq m(\pi_\otimes, \mathbf{r}_{\alpha_\pi}(\sigma)).
\end{equation}

This inequality will play a major role in our arguments, since, as will be shown in some cases, the computation of its righthand side is susceptible to combinatorial methods.

In particular, the following mechanism will allow us to use inductive arguments on multiplicities of $\pi_\otimes$ in Jacquet modules.

Let us write $\pi'= L(\gotM -\gotM\cdot\mathbbm{1}_{S_1})\in \irr(G_m)$. Then, $\pi_\otimes = L(\gotM\cdot\mathbbm{1}_{S_1})\otimes \pi'_\otimes\in \irr(M_{\alpha_\sigma})$.

Note that $\mathbf{r}_{\alpha_\pi} = (1\otimes \mathbf{r}_{\alpha_{\pi'}})\circ\mathbf{r}_{(n-m,m)}$. Thus, from Proposition \ref{prop-geom} we deduce that
\begin{proposition}\label{prop-count-mul}
\begin{equation}
m(\pi_\otimes,\,\mathbf{r}_{\alpha_\pi} (\pi_1\times\cdots\times \pi_r)) =
 \sum_{s\in \mathcal{J}(\pi_1,\ldots,\pi_r)} m(L(\gotM\cdot\mathbbm{1}_{S_1}),\,\tau^s)\;\cdot\; m(\pi'_\otimes,\,\mathbf{r}_{\alpha_{\pi'}} (\delta^s))\;,
\end{equation}
\end{proposition}
where $\mathbf{r}_{\alpha_{\pi'}}$ should be read as the zero functor when applied on representations of groups other than $G_m$.

\subsection{Permutation indexing}
We want to introduce certain ``coordinate systems" on the collection of multisegments $\mathbb{N}(\seg_{\langle \rho \rangle})$.

Let $\mathcal{P}_n$ denote the collection of tuples $(\lambda_1,\ldots,\lambda_n)\in\mathbb{Z}^n$, for which $\lambda_1\leq \lambda_2 \leq \ldots \leq \lambda_n$. It will sometimes be convenient to think of $\mathcal{P}_n$ as elements of $\mathbb{N}(\mathbb{Z})$, in the sense that $\lambda = \sum_{i=1}^n \mathbbm{1}_{\lambda_i}$. In particular, when $\lambda'\leq \lambda$ for $\lambda\in \mathcal{P}_n$ and $\lambda'\in \mathcal{P}_{m}$, we can write $\lambda-\lambda'\in \mathcal{P}_{n-m}$. That would mean the tuple constructed from $\lambda$ after removing the entries of $\lambda'$.

Let $\lambda = (\lambda_1,\ldots,\lambda_n)$, $\mu = (\mu_1,\ldots,\mu_n)\in \mathcal{P}_n$ be given. Let $w\in S_n$ be a permutation for which $\lambda_i \leq \mu_{w(i)}+1$ holds for all $1\leq i\leq n$. For such parameters we define a multisegment
\[
\gotM^w_{\lambda,\mu} = \sum_{i=1}^n [\lambda_i, \mu_{w(i)}]\in \mathbb{N}(\seg_{\langle \rho \rangle})\;.
\]
\begin{remark}
Such a presentation for a given multisegment is not unique. For example, a single segment $\gotM = \Delta = [0,1]$ can be written both as $\gotM^{e}_{(0),(1)}$ and as $\gotM^w_{(0,1),(0,1)}$, where $e$ is the identity in $S_1$ and $w$ is the transposition in $S_2$.
\end{remark}

The group $S_n$ naturally acts on tuples of $n$ integers. Let $S^\lambda$ denote the stabilizer of such tuple $\lambda$, which is clearly a parabolic subgroup of $S_n$ when viewed as a Coxter group (that is, a naturally embedded subgroup of the form $S_{n_1}\times \cdots \times S_{n_t}$). It is easy to see that $\gotM^w_{\lambda,\mu}$ is well-defined by the double-coset of $w$ in $S^\mu \setminus S_n / S^\lambda$.

Recall that $S_n$ is partially ordered by the \textit{Bruhat order}, which we will denote as $\leq$. See, for example, \cite[Section 2]{Hender} for a combinatorial description.

We write $Q(\lambda, \mu) \subset S_n$ for the set of permutations $w$, for which $\gotM^w_{\lambda,\mu}$ is defined ($\lambda_i \leq \mu_{w(i)}+1\;\forall i$). It is known \cite{Hender} that $Q(\lambda,\mu)$ is a lower ideal in the partially ordered set $(S_n, \leq)$. In other words, for any $w\in Q(\lambda,\mu)$ and $x\in S_n$ with $x\leq w$, we have $x\in Q(\lambda,\mu)$.

Let $S(\lambda,\mu)\subset Q(\lambda,\mu)$ be the set of permutations $w$ which are maximal in $S^\mu w S^\lambda$ with respect to the Bruhat order.

It is known that for any $x\in S(\lambda,\mu)$, the set of isomorphism classes of irreducible subquotients which appear in the standard module $M(\gotM^x_{\lambda,\mu})$ is given precisely by
\[
\{L(\gotM^w_{\lambda,\mu})\;:\; w\in S(\lambda,\mu),\, w\geq x\}\;.
\]

With this indexing of multisegments in hand, let us make the parabolic induction problem more precise.

\begin{proposition}\label{gen-decomp}
Let $\pi_i = L(\gotM_i)\in \irr_{\langle \rho \rangle}$, $i=1,\ldots, k$ be given, and suppose that $\sum_{i=1}^k \gotM_i = \gotM^x_{\lambda,\mu}$, where $\lambda,\mu\in \mathcal{P}_m$ and $x\in S(\lambda,\mu)$.

Then, the parabolic induction product decomposes in $\mathcal{R}$ as
\[
[\pi_1]\times \cdots \times [\pi_k]= \sum_{x\leq w\in S(\lambda,\mu)} m_w [L(\gotM^w_{\lambda,\mu})]\;,
\]
for some non-negative integers $m_w$.
\end{proposition}
\begin{proof}
Recall that $\pi_i$ is a quotient of the standard module $M(\gotM_i)$, for all $i$. Hence, $\pi_1\times\cdots\times \pi_k$ is a quotient of $M(\gotM_1)\times\cdots\times M(\gotM_k)$. But, in $\mathcal{R}$ we have $[M(\gotM_1)]\times\cdots\times [M(\gotM_k)] = [M(\gotM^x_{\lambda,\mu})]$. The result then follows from the characterization of subquotients of a standard module.
\end{proof}

\subsection{Kazhdan-Lusztig polynomials and a consequence}\label{sect-kl}
Given two permutations $x,w\in S_n$, let $P_{x,w}$ denote the associated Kazhdan-Lusztig polynomial. These are polynomials with integer coefficients which can be defined for any pair of elements of a Coxeter group. See, for example \cite[Chapter 7]{humph-cox}, for the basic construction.

For any $\lambda,\mu\in \mathcal{P}_m$ and $x\in S(\lambda,\mu)$, we have
\begin{equation}\label{kl-pol}
[M(\gotM^x_{\lambda,\mu})] = \sum_{x\leq w\in S(\lambda,\mu)} P_{x,w}(1) [L(\gotM^w_{\lambda,\mu})]\;.
\end{equation}

This formula is elegantly deduced in \cite{suzuki} from the known decomposition of Verma modules of the Lie algebra $\mathfrak{gl}_n$. A more direct approach is supplied by \cite{Hender}, which applies some combinatorial manipulations on geometric results on representations of $GL_n(F)$ which predate Suzuki's proof (see the introduction section of \cite{Hender} for a survey).

We stress that the coefficients in the above decomposition depend only on the permutations $x,w$ and not on $\lambda,\mu$.

Since both irreducible representations and standard modules give bases for $\mathcal{R}$, the coefficients in \eqref{kl-pol} should be thought of as a transition matrix between the two bases.
It is a well-known fact that for all $w\in S(\lambda,\mu)$, $P_{w,w}(1)=m(L(\gotM^w_{\lambda,\mu}), M(\gotM^w_{\lambda,\mu}))=1$. Because of that fact we can apply an inversion process on the equalities \eqref{kl-pol} to obtain the inverse transition matrix. Namely, for any $\lambda,\mu\in \mathcal{P}_m$ and $x\in S(\lambda,\mu)$,
\begin{equation}\label{kl-pol-inv}
[L(\gotM^x_{\lambda,\mu})] = \sum_{x\leq w\in S(\lambda,\mu)} c_{x,w} [M(\gotM^w_{\lambda,\mu})]\;.
\end{equation}
By viewing $\{P_{x,w}(1)\}_{x\leq w}$ as an element of the incidence algebra of the partially ordered set $S(\lambda,\mu)$, we see that $\{c_{x,w}]_{x\leq w}$ is its inverse. From a general formula for such inversion we see that the value $c_{x,w}$ depends only on values of the form $P_{x,z}(1)P_{z,w}(1)$ for $z\in S(\lambda,\mu)$. Thus, the notation for the values $c_{x,w}$ is justified in the sense that they, again, do not depend on $\lambda,\mu$.

This discussion has a curious consequence, which we want to describe here. We will first need to set some notation.

Let $\lambda,\mu\in \mathcal{P}_m$ be given together with partitions $\{1,\ldots,m\} = I_\lambda \dot{\cup} J_\lambda = I_\mu \dot{\cup} J_\mu$, for which $|I_\lambda|=|I_\mu|=k$ with $0<k<m$. We write $I_\lambda = \{i^\lambda_1 < \ldots< i^\lambda_k\}$, $J_\lambda = \{j^\lambda_1<\ldots< j^\lambda_{m-k}\}$ and similarly for $I_\mu, J_\mu$.

%Let us define $z\in S_m$ to be the permutation given by $z(i^\lambda_t) = i^\mu_t$ and $z(j^\lambda_s) = j^\mu_s$, for all $1\leq t\leq k$, $1\leq s\leq m-k$.

Given any two permutations $w_I\in S_k$ and $w_J\in S_{m-k}$, we construct a permutation $\widetilde{w_I\ast w_J}\in S_m$ by
\[
\widetilde{w_I\ast w_J}(i^\lambda_t) = i^\mu_{w_I(t)}\; \forall\,1\leq t\leq k\;,\quad
\widetilde{w_I\ast w_J}(j^\lambda_s) = j^\mu_{w_J(s)}\; \forall\,1\leq s\leq m-k\;.
\]

We write also $\lambda_I = (\lambda_{i^\lambda_t})_{t=1}^k, \mu_I = (\mu_{i^\mu_t})_{t=1}^k \in \mathcal{P}_k$ and $\lambda_J = (\lambda_{j^\lambda_s})_{s=1}^{m-k}, \mu_J = (\mu_{j^\mu_s})_{s=1}^{m-k}\in\mathcal{P}^{m-k}$. Clearly, $\lambda_I + \lambda_J= \lambda$ and $\mu_I + \mu_J = \mu$.

Now, it is easily checked that for all $w_I \in S(\lambda_I,\mu_I)$ and $w_J \in S(\lambda_J,\mu_J)$, we have $\widetilde{w_I\ast w_J} \in Q(\lambda,\mu)$ and
\[
\gotM^{w_I}_{\lambda_I, \mu_I} +  \gotM^{w_J}_{\lambda_J, \mu_J} = \gotM^{\widetilde{w_I\ast w_J}}_{\lambda,\mu}\;.
\]
We write $w_I\ast w_J\in S(\lambda,\mu) $ for the longest representative of the double-coset of $\widetilde{w_I\ast w_J}$. Note, that the construction of $w_I\ast w_J$ out of $w_I,w_J$ depends only on the partitions determined by $I_\lambda, I_\mu$ and the parabolic subgroups $S^\lambda,S^\mu$.

\begin{theorem}\label{thm-indp-len}
Suppose that $\lambda^1,\lambda^2, \mu^1,\mu^2\in \mathcal{P}_m$ are such that $S^{\lambda^1} = S^{\lambda^2}$ and $S^{\mu^1}= S^{\mu^2}$. Suppose that non-trivial partitions $\{1,\ldots,m\} = I_\lambda \dot{\cup} J_\lambda = I_\mu \dot{\cup} I_\mu$ with $|I_\lambda|=|I_\mu|=k$ are given.

Let $w_I \in S(\lambda^1_I, \mu^1_I) \cap S(\lambda^2_I, \mu^2_I)$ and $w_J \in S(\lambda^1_J, \mu^1_J) \cap S(\lambda^2_J, \mu^2_J)$ be two given permutations, and let
\[
\left[L(\gotM^{w_I}_{\lambda^\ell_I,\mu^\ell_I})\right]\times  \left[L(\gotM^{w_J}_{\lambda^\ell_J,\mu^\ell_J})\right]= \sum_{w_I\ast w_J \,\leq z\in S(\lambda^\ell,\mu^\ell)} m^\ell_z \left[L(\gotM^z_{\lambda^\ell,\mu^\ell})\right]\;,\quad\ell=1,2\;,
\]
be the decompositions as in Proposition \ref{gen-decomp}.

Then, for all $w_I\ast w_J \,\leq z\in S(\lambda^1,\mu^1)\cap S(\lambda^2,\mu^2)$,
\[
m^1_w = m^2_w\;.
\]
\end{theorem}
\begin{proof}

It follows from \eqref{kl-pol-inv} that for $\ell =1,2$,
\[
\left[L(\gotM^{w_I}_{\lambda^\ell_I,\mu^\ell_I})\right]\times  \left[L(\gotM^{w_J}_{\lambda^\ell_J,\mu^\ell_J})\right] =
\]
\[
=\sum_{w_I\leq w\in S(\lambda^\ell_I,\mu^\ell_I)}\sum_{w_J\leq w'\in S(\lambda^\ell_J,\mu^\ell_J)} c_{w_I,w}c_{w_J,w'} [M(\gotM^w_{\lambda^\ell_I,\mu^\ell_I})] [M(\gotM^{w'}_{\lambda^\ell_J,\mu^\ell_J})]\;.
\]
Yet, $\gotM^w_{\lambda^\ell_I,\mu^\ell_I} + \gotM^{w'}_{\lambda^\ell_J,\mu^\ell_J} = \gotM^{w\ast w'}_{\lambda^\ell,\mu^\ell}$ means that $[M(\gotM^w_{\lambda^\ell_I,\mu^\ell_I})] [M(\gotM^{w'}_{\lambda^\ell_J,\mu^\ell_J})] = [M(\gotM^{w\ast w'}_{\lambda^\ell,\mu^\ell})]$.

Thus, by applying \eqref{kl-pol} we see that for all $w_I\ast w_J\leq z\in S(\lambda^\ell,\mu^\ell)$, we have
\begin{equation}\label{coeff-form}
 m^\ell_z  = \sum_{w,w'} c_{w_I,w}c_{w_J,w'} P_{w\ast w', z}(1)\;,
\end{equation}
where the sum is over pairs $(w,w')\in S(\lambda^\ell_I,\mu^\ell_I)\times S(\lambda^\ell_J,\mu^\ell_J)$ satisfying $w_J\leq w$, $w_I\leq w'$ and $w\ast w'\leq z$.

Since $Q(\lambda^\ell,\mu^\ell)$ is a lower ideal, for $z\in S(\lambda^\ell,\mu^\ell)$ and for any pair $(w,w')\in S_k \times S_{m-k}$ for which $w\ast w'\leq z$ holds, we have $w\ast w'\in S(\lambda^\ell,\mu^\ell)$. It is easy to check that this implies $w\in S(\lambda^\ell_I,\mu^\ell_I)$ and $w'\in S(\lambda^\ell_J,\mu^\ell_J)$. Hence, the sum in equation \eqref{coeff-form} can be taken over pairs $(w,w')\in S_k \times S_{m-k}$ which are longest in their respective $S^{\lambda^\ell_I}-S^{\mu^\ell_I}$ or $S^{\lambda^\ell_J}-S^{\mu^\ell_J}$ double-coset, and satisfy the three conditions detailed above.

Since the condition $S^{\lambda^1} = S^{\lambda^2}$ clearly implies $S^{\lambda^1_I}=S^{\lambda^2_I}$ and $S^{\lambda^1_J}=S^{\lambda^2_J}$, we conclude that in the case of $w_I\ast w_J \,\leq z\in S(\lambda^1,\mu^1)\cap S(\lambda^2,\mu^2)$ the sum in equation \eqref{coeff-form} does not depend on $\ell$.

\end{proof}

\begin{corollary}\label{cor-add}
Let $\gotM_1, \gotM_2\in \mathbb{N}(\seg_{\langle \rho \rangle})$ be given multisegments. Let $s\in \mathbb{Z}$ be a fixed positive integer, and let $\check{\gotM}_i$, $i=1,2$ be the multisegment constructed from $\gotM_i$ by taking every segment $[a,b]\in \underline{\gotM_i}$ and changing it into $[a,b+s]\in \underline{\check{\gotM}_i}$.

We can then write $\gotM_1 + \gotM_2 = \gotM^x_{\lambda,\mu}$ and $\check{\gotM}_1 + \check{\gotM}_2 = \gotM^x_{\lambda,\check{\mu}}$, for $x\in S(\lambda,\mu)$, where $\check{\mu}$ is constructed from $\mu$ by adding $s$ to all its entries. When decomposing as in Proposition \ref{gen-decomp},
\[
[L(\gotM_1)]\times [L(\gotM_2)]= \sum_{x\leq w\in S(\lambda,\mu)} m_w [L(\gotM^w_{\lambda,\mu})]\;,
\]
\[
[L(\check{\gotM}_1)]\times [L(\check{\gotM}_2)]= \sum_{x\leq w\in S(\lambda,\check{\mu})} \check{m}_w [L(\gotM^w_{\lambda,\check{\mu}})]\;,
\]
we have $m_w = \check{m}_w$ for all $w\in S(\lambda,\mu)$.
\end{corollary}
\begin{proof}
Let us write $\gotM_i = \gotM^{y_i}_{\lambda_i,\mu_i}$, for $i=1,2$. It is clear then that $\check{\gotM}_i = \gotM^{y_i}_{\lambda_i,\check{\mu}_i}$, where $\check{\mu}_i$ are constructed by adding the constant $s$ to the entries of $\mu_i$.

Then, $\lambda = \lambda_1 + \lambda_2$, $\mu= \mu_1 + \mu_2$, $\check{\mu} = \check{\mu}_1 + \check{\mu}_2$ and $x = y_1\ast y_2\in S_m$, where the $\ast$ operation is given by taking partitions of $\{1,\ldots,m\}$ which correspond to the above additive decompositions of $\lambda^1$ and $\mu^1$.

The statement follows from Theorem \ref{thm-indp-len} after noting that $S(\lambda,\mu)\subset S(\lambda,\check{\mu})$.
\end{proof}

\section{Width invariant}\label{sect-wid}

\subsection{Special classes}

\begin{definition}
A representation $\pi\in \irr$ is called a \textit{ladder representation} if $\pi\in \irr_{\langle \rho \rangle}$ for some $\rho\in \mathcal{C}$ and $\pi= L(\gotM)$, where $\gotM=[a_1,b_1]_\rho+\ldots+[a_k,b_k]_\rho\in \mathbb{N}(\seg_{\langle \rho \rangle})$ is such that $a_1<\ldots<a_k$ and $b_1<\ldots<b_k$. In this case we will also call $\gotM$ a ladder multisegment.
\end{definition}

In \cite{ram}, using the setting of affine Hecke algebras it was shown that ladder representations are characterized as the elements of $\irr$ whose Jacquet modules are all completely reducible representations. In fact, that property was taken as the definition of a \textit{calibrated} representation.

A full description of the Jacquet functor of a ladder representation in terms of multisegments was given in \cite{LapidKret}. In particular, all Jacquet module components of a ladder representation are ladder representations themselves. We will apply that description extensively throughout this work. It is more natural to refrain from recalling the general formula now, but rather give it specifically in each case when necessary.

For reference, let us recall here the formula for a simple special case of such a description, which was known much earlier (see \cite{Zel}). This is the case of a single segment representation, whose Jacquet modules for maximal parabolic subgroups is given as follows:
\begin{equation}\label{jac-seg}
\sum_{k=0}^{n}[\mathbf{r}_{(k,n-k)}(L([a,b])] = \sum_{c=a-1}^b [L([c+1,b])\otimes L([a,c])] \;,
\end{equation}
for any $L([a,b])\in \irr_{\langle \rho \rangle}\cap \irr(G_n)$.

\begin{definition}
For $k,n\in \mathbb{Z}$, let $\gotM(k,n)\in \mathbb{N}(\seg_{\langle \rho \rangle})$ be the multisegment given as a sum of $n$ copies of the segment $[k,k]$. We call the representation
\[
K(\rho, k,n) = L(\gotM(k,n)) = (\rho\nu^k)\times \cdots\times (\rho\nu^k) \;(n\mbox{ times })\;\in \irr_{\langle \rho \rangle}
\]
a (generalized) \textit{Kato module}.
\end{definition}

These representations are named after S.I. Kato who handled in \cite{kato-mod} this class in a greater generality for the setting of affine Hecke algerbas. The Jacquet modules constructed from Kato modules are far from being completely reducible.

More precisely, it is shown in \cite[Proposition 3.3]{groj-vaz} that the socle (and co-socle) filtration\footnote{The socle filtration on a representation space $V$ would be the sequence of sub-representations $\{0\}= V_0 \subset V_1 \subset \ldots \subset V_t = V$, where $V_i/V_{i-1}$ is the socle (maximal completely reducible sub-representation) of $V/V_{i-1}$.} of the representation $\mathbf{r}_\alpha(K(\rho, k,n))$ has length $n$, where $\alpha = ((n-1)m,m)$ and $\rho\in \mathcal{C}(G_m)$. In that sense, these irreducible representations lie at the opposite extreme to ladder representations.

In addition, it is also shown that
\begin{equation}\label{kato-for}
[\mathbf{r}_\alpha(K(\rho,k,n))] = n\cdot[K(\rho,k,n-1)\otimes \rho\nu^k]\;.
\end{equation}

\begin{remark}
Since the results of \cite{groj-vaz} were worked out for representations of affine Hecke algebras, they a priori hold only when $\rho$ is the trivial representation of $G_1$. Yet, by well-known considerations of equivalences of categories which were outlined in the introduction section we can deduce the same analysis for our generalized notion of Kato modules.

As pointed out by the anonymous referee, it is also an easy exercise to deduce Equation \eqref{kato-for} directly from the Geometric Lemma (Proposition \ref{prop-geom}).
\end{remark}

\subsection{Invariants}

Given $\pi\in \mathfrak{R}(G_n)$  with $[\pi]\in \mathcal{R}_{\langle \rho \rangle}$, let us write $\kappa(\pi)$ for the maximal number $m$, for which $\pi$ has a Jacquet module component isomorphic to $K(\rho,k,m)$, for some $k$.

From our previous remarks on the Jacquet functor of Kato modules we see that $\pi$ must have a Jacquet module whose socle filtration is of length at least $\kappa(\pi)$. Thus, $\kappa$ can be viewed as a measure of semi-simplicity of Jacquet modules.

We will see that the same invariant $\kappa$ can be defined in several additional methods, which together produce a useful tool for treating our main problem.

\begin{definition}
The \textit{width} $\wid(\gotM)$ of a non-trivial multisegment $\gotM\in \mathbb{N}(\seg_{\langle \rho \rangle})$ is the minimal number $k$, for which it is possible to write $\gotM = \gotM_1 +\ldots + \gotM_k$ for some ladder multisegments $\gotM_1,\ldots,\gotM_k\in \mathbb{N}(\seg_{\langle \rho \rangle}) $.

We also write $\wid(\pi) = \wid(\gotM)$ for the width of the representation $\pi = L(\gotM)\in \irr_{\langle \rho \rangle}$.

For any $\pi\in \mathfrak{R}(G_n)$ with $[\pi]\in \mathcal{R}_{\langle \rho \rangle}$ the definition is extended by
\[
\wid(\pi) = \max\{ \wid(\sigma)\::\:\sigma\in\irr(G_n),\; m(\sigma,\pi)>0\}.
\]

\end{definition}

Note that ladder representations are precisely those irreducible representations with width $1$. Note too that $\wid(\gotM)$ is always bounded by the number $\# \gotM$ of segments in $\gotM$ and can reach equality, e.g. for Kato modules.

Now, let us consider the relation $\preceq'$ on $\seg_{\langle \rho \rangle}$ that is defined by
\[
[a,b]\preceq' [c,d]\quad \Leftrightarrow \quad \mbox{either}\;\left\{\begin{array}{l} a< c\\ b<d\end{array} \right. \;\mbox{or}\; [a,b]=[c,d].
\]
This relation can be viewed as the transitive and reflexive closure of $\prec$. Note that if $\Delta_1,\Delta_2\in \seg_{\langle \rho \rangle}$ are such that $\Delta_1\npreceq' \Delta_2$ and $\Delta_2\npreceq' \Delta_1$, then we must have either $\Delta_1\subseteq \Delta_2$ or $\Delta_2 \subseteq \Delta_1$.

\begin{lemma}\label{dilw}
For every $\gotM\in \mathbb{N}(\seg_{\langle \rho \rangle})$,
\[
\wid(\gotM) = \max\{ k\;:\;  \mbox{there are non-trivial segments }\Delta_1\subseteq \ldots\subseteq \Delta_k\mbox{ s.t. }\Delta_1+\ldots+\Delta_k\leq\gotM \}.
\]
\end{lemma}

\begin{proof}
Note that the collection of segments in $\gotM$, counted with multiplicities, together with the relation $\preceq'$ is a partially ordered set. A chain for this poset would give a ladder multisegment, while an antichain is a multisegment $\Delta_1+\ldots+\Delta_n\leq \gotM$ for which $\Delta_1\subseteq \ldots\subseteq \Delta_n$ holds. Thus, the statement follows from Dilworth's theorem \cite{dilw}.

\end{proof}

\begin{lemma}\label{first-ineq}
For every $\pi\in\mathfrak{R}(G_n)$ with $[\pi]\in \mathcal{R}_{\langle \rho \rangle}$, we have
\[
\wid(\pi) \leq \kappa(\pi).
\]
\end{lemma}
\begin{proof}
From exactness of the Jacquet functor, it suffices to prove the statement for $\pi\in\irr_{\langle \rho \rangle}$.

Let $\pi\in \irr_{\langle \rho \rangle}$ be given. We write $\pi = L(\gotM)$ and $k=\wid(\pi)$. By Lemma \ref{dilw}, there are non-trivial segments $\Delta_1+\ldots+ \Delta_{k}\leq \gotM$ for which $\Delta_1\subseteq \ldots\subseteq \Delta_{k}$.
We write
\[
S = \left\{\Delta\in\seg_{\langle \rho \rangle}\;:\; \Delta_i\preceq' \Delta\;\,\mbox{for some }1\leq i\leq k \right\}.
\]
Let us define the following multisegments:
\[\gotM_1 = \mathbbm{1}_S\cdot\left(\gotM-(\Delta_1+\ldots + \Delta_{k})\right),\]
\[\gotM_2= (1-\mathbbm{1}_S)\cdot\left(\gotM-(\Delta_1+\ldots + \Delta_{k})\right)= \gotM - \gotM_1 - (\Delta_1+\ldots + \Delta_{k}).\]

We claim that the co-standard module satisfies
\[\widehat{M}(\gotM)\cong  \widehat{M}(\gotM_2)\times L(\Delta_1)\times \cdots \times L(\Delta_{k})\times \widehat{M}(\gotM_1).\]
It is enough to check that $\Delta \nprec \Delta'$, for all $\Delta\in \underline{\gotM-\gotM_2}$ and all $\Delta'\in \underline{\gotM-\gotM_1}$.

Assume the contrary for certain $\Delta,\Delta'$. Then $\Delta_{i_0}\preceq' \Delta$ for a certain $i_0$ and from the transitivity of $\preceq'$ we get $\Delta_{i_0}\preceq' \Delta'$. Hence, $\Delta'\not\in \underline{\gotM_2}$ and we must have $\Delta' = \Delta_j$ for some $j$. But, from $\Delta_{i_0}\preceq' \Delta \prec \Delta'= \Delta_j$ we get $\Delta = \Delta' = \Delta_{i_0}$, which is a contradiction to the non-reflexivity of $\prec$.

Since $\pi$ is embedded in $\widehat{M}(\gotM)$, from Frobenius reciprocity we see that the representation $\pi':=L(\Delta_1+\ldots +\Delta_{k})\cong L(\Delta_1)\times \cdots \times L(\Delta_{k})$ is a Jacquet module component of $\pi$.

Let us write $\Delta_i = [a_i,b_i]$ for all $1\leq i\leq k$. Then, $a_1\leq \ldots\leq a_{k}$, $b_1\geq \ldots \geq b_{k}$ and $a_k\leq b_k$. From the formula for Jacquet modules of segment representations and Proposition \ref{prop-geom}, we know that
\[
\mathbf{r}_\alpha(\pi') = L\left(\sum_{i=1}^{k}[a_{k},b_i]\right) \otimes L\left( \sum_{i=1}^{k} [a_i,a_{k}-1]\right),
\]
where $P_\alpha$ is the appropriate parabolic subgroup.

Thus, $\pi''=  L\left(\sum_{i=1}^{k}[a_{k},b_i]\right)$ is a Jacquet module component of $\pi'$. Using again the arguments from the previous paragraph, we see that $K(\rho, a_k,k)$ is a Jacquet module component of $\pi''$. Since the relation of being a Jacquet module component is clearly transitive, $K(\rho,a_k,k)$ becomes a Jacquet module component of $\pi$.

Therefore, $\wid(\pi) = k\leq \kappa(\pi)$.
%Thus, $\wid(\pi) =k= B(\pi'')\leq j(\pi)$.

\end{proof}

For $\sigma\in \irr_{\langle \rho \rangle}$, let $b(\sigma)\in \mathbb{Z}_{\langle \rho \rangle}\cong \mathbb{Z}$ be the minimal element in $\underline{\supp(\sigma)}$. We write $B(\sigma)=\supp(\sigma)(b(\sigma))$ for the multiplicity of $b(\sigma)$ in $\supp(\sigma)$.

For $\pi\in \mathfrak{R}(G_n)$  with $[\pi]\in \mathcal{R}_{\langle \rho \rangle}$, we write
\[
j(\pi) = \max \;\{\; B(\pi')\;:\; \pi'\in \irr\mbox{ is a Jacquet module component of }\pi\}.
\]

Since $B(K(\rho,k,m))=m$, we always have $\kappa(\pi)\leq j(\pi)$.

\begin{proposition}\label{main-prop}
Let $\pi_1,\ldots, \pi_k\in \irr_{\langle \rho \rangle}$ be ladder representations.

Then $\wid(\pi_1\times\cdots\times \pi_k)\leq k$.
\end{proposition}
\begin{proof}
By Proposition \ref{prop-add}, a Jacquet module component $\sigma$ of $\pi_1\times\cdots\times \pi_k$ is a subquotient of $\sigma_1\times\cdots\times \sigma_k$, where each $\sigma_i$ is a Jacquet module component of $\pi_i$. It follows from the main result of \cite{LapidKret}, that every such $\sigma_i$ must be a ladder representation, which means that $B(\sigma_i)=1$. It easily follows that $B(\sigma) = B(\sigma_1\times\cdots\times \sigma_k)\leq k$. Hence, $j(\pi_1\times\cdots\times \pi_k)\leq k$ and the result follows from Lemma \ref{first-ineq}.

\end{proof}

\begin{proposition}\label{width-equiv}
For every $\pi\in\mathfrak{R}(G_n)$ with $[\pi]\in \mathcal{R}_{\langle \rho \rangle}$,
\[
\wid(\pi) = \kappa(\pi) = j(\pi).
\]
\end{proposition}
\begin{proof}
It is enough to consider $\pi\in \irr$. By definition we can write $\pi = L(\gotM_1+\ldots+\gotM_k)$, where $k=\wid(\pi)$ and $\pi_i = L(\gotM_i)$ are ladder representations, for all $1\leq i\leq k$. Thus, $m(\pi,\:\pi_1\times\cdots\times \pi_k)>0$.
From exactness of the Jacquet functor and the proof of Proposition \ref{main-prop}, we get $j(\pi)\leq j(\pi_1\times\cdots\times \pi_k)\leq k$. Combining with Lemma \ref{first-ineq}, the statement now follows.

\end{proof}

Theorem \ref{main-1} now follows from the next corollary.

\begin{corollary}\label{cor-wid}
For all $\pi_1\in\mathfrak{R}(G_{n_1}),\;\pi_2\in \mathfrak{R}(G_{n_2})$ with $[\pi_1],[\pi_2]\in \mathcal{R}_{\langle \rho \rangle}$,
\[
\wid(\pi_1\times\pi_2)\leq \wid(\pi_1)+\wid(\pi_2).
\]
\end{corollary}

\begin{proof}
By exactness of parabolic induction it is enough to assume that $\pi_1,\pi_2\in \irr_{\langle \rho \rangle}$. Suppose that $\pi_1 = L(\gotM_1+\ldots + \gotM_s)$ and $\pi_2=L(\gotM_{s+1}+\ldots+\gotM_t)$ for ladder representations $\tau_i=L(\gotM_i)$, where $1\leq i\leq t$ and $s=\wid(\pi_1),\: t-s = \wid(\pi_2)$. Then from Remark \ref{rmk} it follows that $m(\pi_1,\, \tau_1\times \cdots\times \tau_s)>0$ and $m(\pi_2,\,\tau_{s+1}\times\cdots\times \tau_t)>0$. Thus, again by exactness of parabolic induction $\pi_1\times\pi_2$ appears as a subquotient of $(\tau_1\times\cdots\times \tau_s)\times (\tau_{s+1}\times \cdots\times \tau_t)$. The statement now follows from Proposition \ref{main-prop}.

\end{proof}

\begin{remark}
The above sub-multiplicativity property can also be easily shown using the $\kappa$ characterization of the width invariant.
\end{remark}

Curiously, Proposition \ref{width-equiv} can also be used to recover \cite[Theorem 4.1.(c)]{ram} by using very different tools from the original proof.
\begin{corollary}\label{cor-ram}
For any $\pi\in \irr_{\langle \rho \rangle}\cap \irr(G_n)$, there is a standard Levi subgroup $M_\alpha< G_n$, for which the length of the socle filtration of $\mathbf{r}_\alpha(\pi)$ is at least $\omega(\pi)$.

In particular, the collection of ladder representations in $\irr_{\langle \rho \rangle}$ consists of precisely those representations in $\irr_{\langle \rho \rangle}$ whose Jacquet modules are all completely reducible.

\end{corollary}
\begin{proof}
Since $\omega(\pi)= \kappa(\pi)$, we know that the Kato module $K(\rho, k, \omega(\pi))$ appears as a Jacquet module component of $\pi$. Let us write $\mathbf{r}_\alpha(\pi)$ for the related Jacquet module. Then, from the previously mentioned result in \cite{groj-vaz}, the length of the socle filtration of $\mathbf{r}_\alpha(\pi)$ is at least $\omega(\pi)$.

The last part of the statement follows from the first part together with the observation that all irreducible subquotients of Jacquet modules of ladder representations in formula of \cite{LapidKret} have distinct supercuspidal supports. Hence, Jacquet modules of ladder representations are completely reducible.

\end{proof}

\subsection{Patterns}

We say that a permutation $w\in S_n$ is \textit{flattened into} $\tilde{w}\in S_k$, for $k<n$, if there is a subset $I\subset \{1,\ldots,n\}$ of size $n-k$, such that $\tilde{w}$ is obtained by removing the entries $(i,w(i))_{i\in I}$ and keeping the relative order of all other entries.

We say that a permutation $w\in S_n$ \textit{avoids} (the pattern) $x_1\ldots x_k$, for $\{x_1,\ldots,x_k\}= \{1,\ldots,k\}$, if $w$ cannot be flattened into the permutation $x\in S_k$ given by $x(i)= x_i$.

\begin{lemma}\label{lem-patt}
Let $\lambda = (\lambda_i), \mu=(\mu_i)\in \mathcal{P}_m$ be given, so that $\lambda_m\leq \mu_1$.

Then, for all $w\in S(\lambda,\mu)$ and all $1\leq k\leq m-1$, $w$ avoids the pattern $(k+1)k\ldots 21$, if and only if, $\wid(L(\gotM^w_{\lambda,\mu}))\leq k$.
\end{lemma}
\begin{proof}
Suppose that $w\in S(\lambda,\mu)$ does not avoid $(k+1)k\ldots 21$. Then, there exist $1\leq i_1<\ldots <i_{k+1}\leq m$ such that $w(i_1)>\ldots >w(i_{k+1})$. Hence, the segments
\begin{equation}\label{incl}
[\lambda_{i_{k+1}}, \mu_{w(i_{k+1})}] \subseteq \ldots \subseteq[\lambda_{i_2}, \mu_{w(i_2)}]\subseteq [\lambda_{i_1}, \mu_{w(i_1)}]
\end{equation}
belong to $\gotM^w_{\lambda,\mu}$. This, means $\wid(L((\gotM^w_{\lambda,\mu}))\geq k+1$ by Lemma \ref{dilw}.

Conversely, suppose there exists a sequence of as in \eqref{incl} for some distinct $1\leq i_1,\ldots, i_{k+1}\leq m$ and $w\in S(\lambda,\mu)$. From the definition of $\mathcal{P}_m$ we can assume $i_1<\ldots< i_{k+1}$. Now, from the definition of $S(\lambda,\mu)$ we can deduce that $w(i_{k+1})> \ldots >w(i_1)$.

\end{proof}

Note, that for $\lambda = (\lambda_1< \ldots < \lambda_{m}), \mu = (\mu_1 <\ldots < \mu_{m}) \in \mathcal{P}_{m}$, the multisegment $\gotM^e_{\lambda,\mu}$ for the identity element $e\in S_m$, is a ladder multisegment. In fact, all ladder multisegments can be written in this form.

\begin{corollary}\label{cor-patt}
Let
\[
\pi_1=L(\gotM^e_{\lambda^1,\mu^1}),\ldots, \pi_k=L(\gotM^e_{\lambda^k,\mu^k})\in \irr_{\langle \rho \rangle}
\]
be ladder representations, for some $\lambda^i = (\lambda^i_1< \ldots < \lambda^i_{m_i}), \mu^i = (\mu^i_1 <\ldots < \mu^i_{m_i}) \in \mathcal{P}_{m_i}$, $i=1,\ldots,k$. Write $\lambda = \sum_{i=1}^k \lambda^i$ and $\mu=\sum_{i=1}^k \mu^i$.

Suppose that $\sum_{i=1}^k \gotM^e_{\lambda^i,\mu^i} = \gotM^x_{\lambda,\mu}$ for some $x\in P\left(\lambda,\mu\right)$. Let
\[
[\pi_1]\times\cdots\times [\pi_k]= \sum_{x\leq w\in S(\lambda,\mu)} m_w \left[L\left(\gotM^w_{\lambda,\mu}\right)\right]
\]
be the decomposition as in Proposition \ref{gen-decomp}.

Then, $m_w= 0$ for all $w\in S(\lambda,\mu)$ which do not avoid $(k+1)k\ldots 21$.

%Then, $\pi_1 = L(\gotM^e_{\lambda_1,\mu_1}),\pi_2= L(\gotM^e_{\lambda)\in \irr_{\langle \rho \rangle}$ are ladder representations.

\end{corollary}
\begin{proof}
Let $D$ be a fixed integer. We write $\check{\mu}^i = (\mu^i_j+D)_{j=1}^{m_i}\in \mathcal{P}_{m_i}$ for all $1\leq i\leq k$ and $\check{\mu} = \sum_{i=1}^k \check{\mu}^i$. By choosing $D$ to be big enough we can assume that $\lambda^i_{m_i}\leq \mu^j_1 + D$ for all $1\leq i,j\leq k$, which means $Q(\lambda^i, \check{\mu}^i) = S_{m_i}$ for all $i$, and $Q(\lambda,\check{\mu}) = S_{\sum_i m_i}$.

The representations $\check{\pi}_i = L(\gotM^e_{\lambda^i,\check{\mu}^i})$ remain ladder, and $\sum_{i=1}^k \gotM^e_{\lambda^i,\check{\mu}^i} = \gotM^x_{\lambda,\check{\mu}}$. We write
\[
[\check{\pi}_1]\times\cdots\times [\check{\pi}_k]= \sum_{x\leq w\in S(\lambda,\check{\mu})} \check{m}_w \left[L\left(\gotM^w_{\lambda,\check{\mu}}\right)\right]
\]
for the decomposition as in Proposition \ref{gen-decomp}.

By Proposition \ref{main-prop} and Lemma \ref{lem-patt}, for all irreducible subquotients $L\left(\gotM^w_{\lambda,\check{\mu}}\right)$ of $\check{\pi}_1\times\cdots\times \check{\pi}_k$, the permutation $w$ must avoid the pattern $(k+1)k\ldots 21$. Thus, $\check{m}_w = 0$ for all $w\in S(\lambda,\check{\mu})$ which do not avoid this pattern.

Finally, from a successive application of Corollary \ref{cor-add} we see that $m_w = \check{m}_w$, for all $w\in S(\lambda,\mu)$.

\end{proof}

\section{Multiplicity One}\label{sect-mulone}

In this section we will show that the constants $m_w$ appearing in the decomposition of Proposition \ref{gen-decomp} for the case of a product of two ladder representations cannot be greater than $1$. This will be done by relying on the inequality \eqref{maj} and an analysis of the appearances of indicator representations in Jacquet modules.

\begin{proposition}\label{prop-gen}
Let $\pi_1,\ldots,\pi_k\in \irr$ be given and let $\pi\in\irr$ be a given generic representation.

Then,
\[
m(\pi,\:\pi_1\times \cdots\times \pi_k)\leq 1\;.
\]
In case equality above holds, all $\pi_1,\ldots,\pi_k$ must be generic representations.

\end{proposition}
\begin{proof}
Suppose that $m(\pi,\:\pi_1\times \cdots\times\pi_k)>0$.

Recall that since $\pi$ is generic, its space carries a non-zero Whittaker functional. From exactness of the Whittaker functor, we deduce that $\pi_1\times \cdots \times\pi_k$ carries a non-zero Whittaker functional as well. Now, by Rodier's theorem \cite{rodier} $\pi_1\otimes\cdots\otimes \pi_k$ is generic, hence, so are $\pi_1,\ldots,\pi_k$. Moreover, since $\pi_1,\ldots,\pi_k$ are irreducible it follows from the same theorem that the space of Whittaker functionals on $\pi_1\times\cdots \times\pi_k$ is one-dimensional. Again, by exactness this means $\pi$ cannot appear with multiplicity $>1$ in the product.

\end{proof}

\begin{lemma}\label{mul-lem}
Let $a,b,c$ be integers with $a \leq b$ and $a-1\leq c \leq  b$. Fix the representation $\pi =L([a,b]+[a,c])\in \irr_{\langle \rho \rangle}$, and let $\pi_1=L(\gotM_1), \pi_2=L(\gotM_2)\in \irr_{\langle \rho \rangle}$ be two ladder representations for which $m(\pi,\:\pi_1\times \pi_2)>0 $. Then, $m(\pi,\:\pi_1\times \pi_2)=1$.

In case $c=b$, we must have $\pi_1\cong\pi_2\cong L([a,b])$.

When $c<b$, the pair $\{\gotM_1,\gotM_2\}$ must be of the form
\[
\left\{  \sum_{i=0}^t [a_{2i},a_{2i+1}-1] ,\;  [a,c]+\sum_{i=1}^{s} [a_{2i-1},a_{2i}-1]   \right\},
\]
for some $a=a_0,\; c+1<a_1<\ldots<a_l=b+1$, with either $l=2t+1=2s+1$ or $l=2s= 2t+2$.

\end{lemma}

\begin{proof}
Since the pair of segments $\{[a,b], [a,c]\}$ is not linked, $\pi$ is a generic representation. By Proposition \ref{prop-gen}, multiplicity one follows together with the fact that $\pi_1,\pi_2$ are generic. Since $\gotM_i,\,i=1,2$ are multisegments for which $L(\gotM_i)$ are generic ladder representations, their segments must be pairwise unlinked.

Thus, we can write $\gotM_i = [t^i_0,t^i_1-1] + [t^i_2,t^i_3-1] + \ldots + [t^i_{2k_i},t^i_{2k_i+1}-1]$, for $t^i_0<t^i_1<\ldots <t^i_{2k_i+1},\;i=1,2$.

Note that,
\[
2\supp(L([a,c])) \leq \supp(L([a,c])+\supp(L[a,b]) = \supp(\pi) =
\]
\[= \supp(\pi_1\times \pi_2) = \supp(\pi_1)+\supp(\pi_2)\;.
\]
Clearly it follows that $[a,c]\subseteq [t^1_{2j_1}, t^1_{2j_1+1}-1] $ and $[a,c]\subseteq [t^2_{2j_2}, t^2_{2j_2+1}-1]$ for some $j_1,j_2$.

The rest of the statement easily follows after noting that each of the supercuspidals appearing in $\underline{\supp(L([c+1,b]))}$ must appear only once in $\supp(\pi_1)+\supp(\pi_2)$.

\end{proof}

\begin{proposition}\label{mult-one}
Let $\pi_1=L(\gotM_1),\pi_2=L(\gotM_2)$ be ladder representations in $\irr_{\langle \rho \rangle}$. Suppose that $\pi_1\times\pi_2\in \mathfrak{R}(G_n)$.

For any $\sigma\in \irr(G_n)$, we have $m(\sigma_\otimes,\,\mathbf{r}_{\alpha_\sigma} (\pi_1\times\pi_2))\leq1$.
\end{proposition}
\begin{proof}
Let $\sigma = L(\gotN)\in \irr(G_n)$. We will prove the statement by induction on the number $\# \gotN$ of segments in $\gotN$.

Let $\Delta\in\underline{\gotN}$ be a segment with minimal $b(\Delta)$. Note that all irreducible subquotients of $\pi_1\times\pi_2$ must have the same supercuspidal support, i.e.~$\supp(L(\gotM_1+\gotM_2))$. We may assume that $\supp(\sigma)=\supp(L(\gotM_1+\gotM_2))$, for otherwise the statement is trivially true. Thus, $b(\Delta)$ is also the minimal point in $\underline{\supp(L(\gotM_1+\gotM_2))}$.

Since $\gotM_1,\gotM_2$ are ladders, $\supp(\sigma)(b(\Delta))\leq2$. Hence, if $m(\sigma_\otimes,\,\mathbf{r}_{\alpha_\sigma} (\pi_1\times\pi_2))>0$, then there can be at most one segment $\hat{\Delta}\in\underline{\gotN-\Delta}$ with $b(\hat{\Delta})=b(\Delta)$. In case there is no such segment, let us still write $\hat{\Delta} = [b(\Delta), b(\Delta)-1]$. Moreover, let us always assume that $e(\hat{\Delta})\leq e(\Delta)$.

 We write $\sigma'= L(\gotN -\Delta-\hat{\Delta})$. Then, $\sigma_\otimes = L(\Delta+\hat{\Delta})\otimes \sigma'_\otimes\in \irr(M_{\alpha_\sigma})$.
%Let us write $M_\beta\cong G_{m_1}\times G_{m_2}$ for the standard Levi subgroup of $G_n$ corresponding to $L(\Delta+\hat{\Delta})\otimes \sigma'$. Then $M_{\alpha_\sigma}$ can be written as $G_{m_1}\times M_{\alpha_{\sigma'}}$.
By Proposition \ref{prop-count-mul}, we have
\[
m(\sigma_\otimes,\,\mathbf{r}_{\alpha_\sigma} (\pi_1\times\pi_2)) = \sum_{(j_1,j_2) \in\mathcal{J}(\pi_1,\pi_2) } m(L(\Delta+\hat{\Delta}),\,\tau_1^{j_1}\times \tau_2^{j_2})\cdot m(\sigma'_\otimes,\,\mathbf{r}_{\alpha_{\sigma'}} (\delta_1^{j_1}\times \delta_2^{j_2})).
\]
%By the Geometric Lemma (see again \cite[Section 1.2]{LM2}), we know that
%\[
%[\mathbf{r}_\beta(\pi_1\times \pi_2)] = \sum_i [\tau_1^i\times \tau_2^i]\otimes [\delta_1^i\times \delta_2^i],
%\]
where $(j_1,j_2)\in \mathcal{J}(\pi_1,\pi_2)$ indexes all possible irreducible subquotients $\tau_1^{j_1}\otimes \delta_1^{j_1}$ and  $\tau_2^{j_2}\otimes \delta_2^{j_2}$ of the appropriate Jacquet modules of $\pi_1$, $\pi_2$, respectively.

Recall that from the result of \cite{LapidKret} it follows that $\tau_i^j, \delta_i^j$ are all ladder representations. It also follows from this result that $\tau_i^j\cong \tau_i^{j'}$, if and only if, $j=j'$.

By the induction hypothesis it is enough to show that $m(L(\Delta+\hat{\Delta}),\,\tau_1^{j_1}\times \tau_2^{j_2})=0$ for all $(j_1,j_2)$, except for possibly one index $(j_1^0,j_2^0)$, and that $m(L(\Delta+\hat{\Delta}),\,\tau_1^{j_1^0}\times \tau_2^{j_2^0})\leq1$. The latter statement indeed follows from Lemma \ref{mul-lem}.

Suppose that the Jacquet modules of $\pi_1,\pi_2$ have respective irreducible subquotients $\tau_1\otimes \delta_1$,  $\tau_2\otimes \delta_2$, for which $\tau_1\times\tau_2$ contains $L(\Delta+\hat{\Delta})$ as a subquotient. It remains to show the uniqueness of such subquotients.

Note first that if $\Delta=\hat{\Delta}$, then Lemma \ref{mul-lem} implies that $\tau_1\cong\tau_2\cong L(\Delta)$, which gives the wanted conclusion. Thus, we can assume $e(\hat{\Delta})< e(\Delta)$ henceforth.

Lemma \ref{mul-lem} now states that such pair $\{\tau_1,\tau_2\}$ must be of the form $\mathcal{T}=\{\rho_1,\rho_2\}$, where

\[
\rho_1= L\left( \sum_{i=0}^t [a_{2i},a_{2i+1}-1] \right),\quad \rho_2 = L\left( \hat{\Delta}+\sum_{i=1}^s  [a_{2i-1},a_{2i}-1] \right),
\]
for some $a=a_0,\; e(\hat{\Delta})+1 <a_1<\ldots<a_l=b+1$, with either $l=2t+1=2s+1$ or $l=2s= 2t+2$. Here we write $\Delta= [a,b]$.

Let us write $\pi_i = L(\Delta^i_1+\ldots+\Delta^i_k)$ with $e(\Delta^i_{k_i})<\ldots<e(\Delta^i_1)$, for $i=1,2$. Recall from \cite{LapidKret} that $\tau_i$, being a leftmost Jacquet module component, must be expressible in the form $L\left( \sum_{j=1}^{k_i} [c^i_j,e(\Delta^i_j)] \right)$, where $c^i_{k_i}<\ldots<c^i_1$ and $b(\Delta^i_j)\leq c^i_j \leq e(\Delta^i_j)+1$ for all $1\leq j\leq k_i$ and $i=1,2$.

Let us first consider the case that $\hat{\Delta}$ is a non-trivial segment. Then, $\supp(L(\gotM_1+\gotM_2))(b(\Delta))=2$ and $c^1_{k_1} = c^2_{k_2}=b(\Delta)$. Comparing the descriptions of $\mathcal{T}$ with that of $\{\tau_1,\tau_2\}$, we see that $\{e(\Delta^1_{k_1}),e(\Delta^2_{k_2})\} = \{e(\hat{\Delta}), a_1-1\}$. Since $e(\hat{\Delta})<a_1-1$, there is a unique identification between $\mathcal{T}$ and $\{\tau_1,\tau_2\}$. Hence, we can assume without loss of generality that $\tau_1=\rho_1$ and $\tau_2=\rho_2$.

Now, in the other case, that is when $\hat{\Delta}$ is the trivial segment, $\supp(L(\gotM_1+\gotM_2))(b(\Delta))=1$. Without loss of generality we can assume that $m(b(\Delta),\,\supp(\pi_1))=1$. Clearly, we again have $\tau_1=\rho_1$ and $\tau_2=\rho_2$.

Assume now there is a different pair $\tau'_1\otimes \delta'_1$,  $\tau'_2\otimes \delta'_2$ of irreducible subquotients of the respective Jacquet modules of $\pi_1,\pi_2$, such that
\[
\tau'_1= L\left( \sum_{i=0}^{t'} [a'_{2i},a'_{2i+1}-1] \right),\quad \tau'_2 = L\left(\hat{\Delta}+ \sum_{i=1}^{s'} [a'_{2i-1},a'_{2i}-1] \right),
\]
with similar assumptions on the indices.

Let $1\leq i_0$ be the minimal index for which $a_{i_0}\neq a'_{i_0}$. We can assume that $a'_{i_0}< a_{i_0}$. Also, without loss of generality we can assume that $i_0$ is odd. Otherwise, $\pi_2$ should replace $\pi_1$ for the rest of the argument.

Note that $a_{i_0}-1 = e(\Delta^1_{j_0})$ and $a_{i_0-1} = c^1_{j_0}$ for some $j_0$. Using a similar reasoning on $\tau'_1$, we see there is an index $j_1$ for which $a'_{i_0}-1 = e(\Delta^1_{j_1})$. Hence, $e(\Delta^1_{j_1})<e(\Delta^1_{j_0})$ and $j_0<j_1$. Therefore,
\[
c^1_{j_1} < c^1_{j_0} = a_{i_0-1} = a'_{i_0-1} < a'_{i_0} = e(\Delta^1_{j_1})+1.
\]
In particular, $[c^1_{j_1},e(\Delta^1_{j_1})]$ is a non-trivial segment in the multisegment defining $\tau_1$, which means $a'_{i_0} = a_{i_1}$ for some odd $i_1<i_0$. But, since $i_1< i_0-1$ we get a contradiction from $a_{i_1}<a_{i_0-1}<a'_{i_0}$.

\end{proof}

\begin{corollary}[Theorem \ref{main-2}]\label{cor-main-2}
Given ladder representations $\pi_1,\pi_2\in \irr$, $m(\sigma,\,\pi_1\times \pi_2)\leq1$ for all $\sigma\in \irr$.
\end{corollary}

\begin{proof}
Suppose that $\pi_1\in \irr_{\langle\rho_1\rangle}$ and $\pi_2\in \irr_{\langle \rho_2\rangle}$. If $\rho_1\not\in \mathbb{Z}_{\langle \rho_2\rangle}$ then $\pi_1\times \pi_2$ is irreducible and there is nothing to prove. Therefore we are free to assume $\pi_1,\pi_2\in \irr_{\langle \rho \rangle}$.

The statement now follows from Proposition \ref{mult-one} and inequality \eqref{maj}.

\end{proof}

\section{More on Indicator Representations}\label{sect-some-more}

Recall from Section \ref{sect-kl} that for $\lambda,\mu \in \mathcal{P}_m$ and $x,w\in S(\lambda,\mu)$, it is known that $m(L(\gotM^w_{\lambda,\mu}),M(\gotM^x_{\lambda,\mu}))$ is positive, if and only if, $x\leq w$ in the Bruhat order on $S_m$.

Although $m(L(\gotM^w_{\lambda,\mu})_\otimes,\mathbf{r}(M(\gotM^x_{\lambda,\mu})))$ is potentially larger than $m(L(\gotM^w_{\lambda,\mu}),M(\gotM^x_{\lambda,\mu}))$, it will be useful (in particular, for the main result of Section \ref{sect-smooth}) to see that the mentioned non-vanishing condition still holds for multiplicities of indicator representations inside Jacquet modules.

Our consideration will involve a view on standard modules as induced from segment representations. Since there is a simple formula for the Jacquet functor of segment representations, the Jacquet modules of standard modules are more accessible to computations than those of irreducible representations.

\begin{lemma}\label{lem-tech}
Let $\lambda= (\lambda_i),\mu= (\mu_i)\in \mathcal{P}_m$ and $w\in S(\lambda,\mu)$ be given.
%We write $\gotM^w_{\lambda,\mu} = \sum_{i=1}^n \Delta_i\in \mathbb{N}(\seg_{\langle \rho \rangle})$ with non-trivial $\Delta_i\in \seg_{\langle \rho \rangle}$.

Let $\pi = L(\sum_{i=1}^k[\lambda_1,e_i])\in \irr_{\langle \rho \rangle}$ be a representation given by the integers $\lambda_1\leq e_1\leq\ldots\leq e_k$.

Suppose that there is an index $s\in \mathcal{J}(L([\lambda_1,\mu_{w(1)}]),\ldots,L([\lambda_m,\mu_{w(m)}]))$ for which $m(\pi,\tau^s)>0$.

Then, $\lambda_1 = \ldots = \lambda_k$, $e_i = \mu_{j_i}$ for all $1\leq i\leq k$ and some $j_1<\ldots<j_k$, and
\[
[\delta^s] = [M(\gotM^x_{\check{\lambda},\check{\mu}})]\;,
\]
where $\check{\lambda} = (\lambda_{k+1},\ldots,\lambda_m)$, $\check{\mu}=\mu- (\mu_{j_1},\ldots,\mu_{j_k})$, and $x\in S(\check{\lambda},\check{\mu})$.

Moreover, when writing $x$ as a bijection $\underline{x}:\{k+1,\ldots, m\}\to \{1,\ldots,m\}\setminus \{j_1,\ldots,j_k\}$, we have $\underline{x}(i) = w(i)$ for all $i$ such that $\lambda_i>\lambda_1$ and $\mu_{w(i)}>e_k$.

\end{lemma}

\begin{proof}
Recall that irreducible subquotients of Jacquet modules of a segment representation $L([a,b])$ are all given in the form $L([c+1,b])\otimes L([a,c])$, for some $a-1\leq c\leq b$. Hence, there are integers $\lambda_i-1 \leq c_i\leq \mu_{w(i)}$ for which
\[
[\tau^s] = \left[M\left(\sum_{i=1}^m [c_i +1, \mu_{w(i)}] \right)\right],\quad [\delta^s] = \left[M\left(\sum_{i=1}^m [\lambda_i,c_i] \right)\right]
\]
hold in $\mathcal{R}$. We write $I = \{ i\,:\, c_i<\mu_{w(i)}\}$, $\Delta_i= [c_i+1,\mu_{w(i)}]$ for all $i\in I$, and $I' = \{i\in I\,:\, b(\Delta_i)=\lambda_1\}$.

Since $\supp(\pi) = \supp(\tau^s)$, it is easy to see that we must be able to write
\[
\sum_{i\in I}\Delta_i = \sum_{t=1}^k \sum_{j=1}^{l_i-1} [d^t_{j},d^t_{j+1}-1]\;,
\]
where $\lambda_1 = d^t_1<d^t_2<\ldots< d^t_{l_i} = e_i+1$, for all $1\leq t\leq k$.

We define a function $\tilde{z}:I\setminus I' \to I$ in the following manner: For $i\in I\setminus I'$ let $(t,j)$, $1<j$, be the index for which $\Delta_i = [d^t_j,d^t_{j+1}-1]$. Then, $\tilde{z}(i)$ is the index for which $\Delta_{\tilde{z}(i)}= [d^t_{j-1},d^t_{j}-1]$. Note, that
\[
c_i = b(\Delta_i)-1 = e(\Delta_{\tilde{z}(i)}) = \mu_{w\tilde{z}(i)}\;.
\]

Let us extend $\tilde{z}$ to $z:\{1,\ldots,m\}\setminus I'\to \{1,\ldots,m\}$ by $z(i)=i$ for $i\not\in I$. We can now write
\[
[\delta^s] = \left[M\left(\sum_{i\not\in I'} [\lambda_i,\mu_{wz(i)}] \right)\right]\;.
\]
The statement readily follows with $\underline{x} = wz$, after identifying $\{k+1,\ldots,m\}$ with $\{1,\ldots,m\}\setminus I'$.

\end{proof}

\begin{lemma}\label{lem-bruhat-spe}
Let $\lambda,\mu\in \mathcal{P}_m$ be given. Suppose that $x,w\in S(\lambda,\mu)$ are two permutations for which $x\nleq w$ in the Bruhat order of $S_m$. Then, there exist $1\leq i,j\leq m$ for which

\begin{equation}\label{condi}
|\left\{1\leq i'\leq i\::\: x(i')\geq j\right\}| > \left|\left\{1\leq i'\leq i\::\: w(i')\geq j\right\}\right|
\end{equation}
holds, and such that $\lambda_i<\lambda_{i+1}$ or $i=m$, and $\mu_{j-1}<\mu_j$ or $j=1$.
\end{lemma}

\begin{proof}
Let us call a pair of indices $1\leq i,j \leq m$ satisfying condition \eqref{condi} an admissible pair. The existence of an admissible pair for $x\nleq w$ is a well-known criterion, which is even taken for the definition of the Bruhat order on permutation groups in some sources. Let us start by choosing an arbitrary admissible pair $(i_0,j_0)$.

Let $i_1$ be the minimal index for which $(i_1,j_0)$ is admissible. It follows that
\[
\left|\left\{1\leq i'\leq i_1-1\::\: w(i')\geq j_0\right\}\right| = \left|\left\{1\leq i'\leq i_1-1\::\: x(i')\geq j_0\right\}\right| =
\]
\[
=\left|\left\{1\leq i'\leq i_1\::\: x(i')\geq j_0\right\}\right|-1\;.
\]
Hence, $w(i_1)<j_0$. Let $i_2\geq i_1$ be such that $\lambda_{i_1} = \lambda_{i_1+1} = \ldots = \lambda_{i_2} < \lambda_{i_2}+1$ or $i_2=m$. Since $w$ is a permutation of maximal length in $w S^\lambda$, we have $j_0> w(i_1)> w(i_1+1)>\ldots > w(i_2)$. Thus,
\[
\left|\left\{1\leq i'\leq i_2\::\: w(i')\geq j_0\right\}\right| = \left|\left\{1\leq i'\leq i_1\::\: w(i')\geq j_0\right\}\right|<
\]
\[
<\left|\left\{1\leq i'\leq i_1\::\: x(i')\geq j_0\right\}\right| \leq \left|\left\{1\leq i'\leq i_2\::\: x(i')\geq j_0\right\}\right|
\]
and $(i_2,j_0)$ is admissible.

Now, let $j_1$ be the maximal index for which $(i_2,j_1)$ is admissible. By a similar argument as before we see from maximality that $w^{-1}(j_1)>i_2$.

Let $j_2\leq j_1$ be such that $\mu_{j_1} = \mu_{j_1-1} = \ldots = \mu_{j_2} > \mu_{j_2-1}$ or $j_2=1$. Since $w$ is a permutation of maximal length in $S^\mu w$, we have $w^{-1}(j_2) > \ldots> w^{-1}(j_1-1) > w^{-1}(j_1)> i_2$. Thus,
\[
\left|\left\{1\leq i'\leq i_2\::\: w(i')\geq j_2\right\}\right| = \left|\left\{1\leq i'\leq i_2\::\: w(i')\geq j_1\right\}\right|<
\]
\[
<\left|\left\{1\leq i'\leq i_2\::\: x(i')\geq j_1\right\}\right| \leq \left|\left\{1\leq i'\leq i_2\::\: x(i')\geq j_2\right\}\right|
\]
and $(i_2,j_2)$ is admissible.

\end{proof}

\begin{proposition}\label{bruhat-prop}
Let $\lambda,\mu\in \mathcal{P}_m$ be given. Suppose that
\[
m(\pi_{\otimes}, \mathbf{r}_{\alpha_\pi}(M(\gotM^x_{\lambda,\mu}))>0
\]
holds, where $\pi = L(\gotM^w_{\lambda,\mu})$, for some $w,x\in S(\lambda,\mu)$.

Then $x\leq w$ in the Bruhat order of $S_m$.

\end{proposition}
\begin{proof}
Suppose that $x\nleq w$. Let $1\leq i_0,j_0\leq m$ be the indices given by Lemma \ref{lem-bruhat-spe} in this case. Consider the representations
\[
\pi_1 = L\left( \sum_{i=1}^{i_0} [\lambda_i,\mu_{w(i)}] \right),\quad \pi_2 = L\left( \sum_{i=i_0+1}^{m} [\lambda_i,\mu_{w(i)}] \right)\;.
\]
Note that $\pi_\otimes =(\pi_1)_\otimes\otimes (\pi_2)_\otimes$.
It follows from the assumption $m(\pi_{\otimes}, \mathbf{r}_{\alpha_\pi}(M(\gotM^x_{\lambda,\mu}))>0$ that there is $s\in \mathcal{J}(L([\lambda_1,\mu_{x(1)}]),\ldots,L([\lambda_m,\mu_{x(m)}]))$, for which $m((\pi_1)_\otimes, \tau^s), m((\pi_2)_\otimes, \delta^s)>0$ hold.

Again, from \eqref{jac-seg}, there are integers $\lambda_i \leq c_i\leq \mu_{x(i)}+1$ for which
\[
[\tau^s] = \left[M\left(\sum_{i=1}^m [c_i , \mu_{x(i)}] \right)\right],\quad [\delta^s] = \left[M\left(\sum_{i=1}^m [\lambda_i,c_i-1] \right)\right]
\]
hold in $\mathcal{R}$. Since $\supp(\pi_2)$ must equal the supercuspidal support of the irreducible subquotients of $\delta^s$ and (unless $\pi_2$ is the trivial representation of $G_0$) $b(\pi_2) \geq \lambda_{i_0+1}> \lambda_i$ for all $1\leq i\leq i_0$, we must have $c_i = \lambda_i$ for all $1\leq i\leq i_0$.

Let $l$ be the maximal index for which $\lambda_1 =\lambda_l$. Define $\sigma = L(\sum_{i=1}^l [\lambda_1, \mu_{w(i)}])$. Then we can write $(\pi_1)_\otimes = \sigma \otimes (\pi'_1)_\otimes$. Since $m((\pi_1)_\otimes, \tau^s)>0$, by Proposition \ref{prop-count-mul} there must be $s'\in \mathcal{J}(L([c_1, \mu_{x(1)}]), \ldots, L([c_m, \mu_{x(m)}]))$ for which $m(\sigma, \tau^{s'})$ and $m((\pi'_1)_\otimes, \delta^{s'})$ are positive.

Write $I= \{1\leq i\leq l\,:\, \lambda_1 \leq \mu_{w(i)}\}$. By Lemma \ref{lem-tech} we have $[\delta^{s'}] = [M(\gotM^{x'}_{\lambda', \mu'})]$, where $\lambda' = (\lambda_{|I|+1}, \ldots, \lambda_m)$, $\mu' = \mu-(\mu_{w(i)})_{i\in I}$, and $x'\in S(\lambda', \mu')$.

For all $i\in \{1,\ldots,l\}\setminus I$, we have $\mu_{w(i)} = \mu_{x'(i')}$ for some $i'$. Since this forces $\lambda_{i'}= \lambda_1$ and $[\lambda_{i'}, \mu_{x'(i')}]$ is a trivial segment, we can furthermore write
\[
[\delta^{s'}] = [M(\gotM^{x''}_{\lambda^{''}, \mu^{''}})]\;,
\]
where $\lambda'' = (\lambda_{l+1}, \ldots, \lambda_m)$, $\mu''=\mu-(\mu_{w(i)})_{i=1}^l$, and $x''\in S(\lambda'', \mu'')$.

Writing $(\pi'_1)_\otimes = \sigma' \otimes (\pi^{''}_1)_\otimes$ and proceeding similarly we arrive at the situation in which $\gotM^{\overline{x}}_{\overline{\lambda},\overline{\mu}}=0$, where $\overline{\lambda} = (\lambda_{i_0+1},\ldots,\lambda_m)$, $\overline{\mu}= (\overline{\mu}_{i_0+1},\ldots, \overline{\mu}_m)= \mu- (\mu_{w(i)})_{i=1}^{i_0} $ , and $\overline{x}\in S(\overline{\lambda}, \overline{\mu})$. Thus, $\overline{\mu}_i = \lambda_{\overline{x}(i)}-1$, for all $i_0 < i\leq m$.

Consider the quantity $Q = |\{i_0<i\leq m\,:\, \overline{\mu}_i \geq \mu_{j_0}\}|$. Note, that
\[
Q = |\{i_0<i\leq m\,:\, \lambda_i \geq \mu_{j_0}+1\}|\leq
\]
\[
\leq |\{i_0< i\leq m \,:\, \mu_{x(i)} \geq \mu_{j_0}\}| = m- (j_0-1) - |\{1\leq i\leq i_0 \,:\, \mu_{x(i)} \geq \mu_{j_0}\}|\;,
\]
with the inequality coming from $\lambda_i\leq \mu_{x(i)}+1$ and the last equality is due to the assumed fact $\mu_{j_0-1}< \mu_{j_0}$. On the other hand, by construction of $\overline{\mu}$,
\[
Q = m-(j_0-1) - |\{1\leq i\leq i_0 \,:\, \mu_{w(i)} \geq \mu_{j_0}\}|\;.
\]
In particular,
\[
|\{1\leq i\leq i_0 \,:\, x(i) \geq j_0\}| = |\{1\leq i\leq i_0 \,:\, \mu_{x(i)} \geq \mu_{j_0}\}|\leq
\]
\[
\leq |\{1\leq i\leq i_0 \,:\, \mu_{w(i)} \geq \mu_{j_0}\}| = |\{1\leq i\leq i_0 \,:\, w(i) \geq j_0\}|
\]
gives a contradiction to the assumptions on $(i_0,j_0)$.

\end{proof}

\section{Conjectures}\label{sect-conj}

Let us focus on the problem of determining the irreducible factors of a product of two ladder representations in $\mathcal{R}$.

Suppose that $\pi_1, \pi_2\in\irr_{\langle \rho \rangle}$ are given ladder representations. We can uniquely write $\pi_1=L(\gotM^e_{\lambda_1,\mu_1}), \pi_2=L(\gotM^e_{\lambda_2,\mu_2})$, for some $\lambda^i = (\lambda^i_1< \ldots < \lambda^i_{m_i}), \mu^i = (\mu^i_1 <\ldots < \mu^i_{m_i}) \in \mathcal{P}_{m_i}$, $i=1,2$, such that $\lambda^i_j\leq \mu^i_j$ for all $i,j$. Denote $\lambda = \lambda_1 + \lambda_2$ and $\mu = \mu_1+\mu_2$.

We write $x\in S(\lambda, \mu)$ for the permutation which satisfies $\gotM^e_{\lambda_1,\mu_1}+ \gotM^e_{\lambda_2,\mu_2} = \gotM^x_{\lambda,\mu}$.

In light of Corollaries \ref{cor-patt} and \ref{cor-main-2}, we can now write
\[
[\pi_1] \times [\pi_2] = \sum_{w\in \mathcal{I}(\pi_1,\pi_2)} \left[L\left(\gotM^w_{\lambda,\mu}\right)\right]\;,
\]
for a subset
\[
\mathcal{I}(\pi_1,\pi_2)\subset \{w\in S(\lambda, \mu)\;:\; x\leq w\,,\,w\mbox{ avoids }321\}\subset S_{m_1+m_2}\;.
\]
Thus, the determination of the sets $\mathcal{I}(\pi_1,\pi_2)$ is the remaining piece of information required to solve the full multiplicity problem for the case of two ladder representations.

In this section we would like to present and discuss a conjectural description of these permutation sets.

\begin{conjecture}\label{main-conj}
Let $\pi_1, \pi_2\in\irr_{\langle \rho \rangle}$ be ladder representations with notation as above.

Then, for all $\sigma= L(\gotM^w_{\lambda,\mu})$ with $w\in S(\lambda,\mu)$, we have that
\[
m(\sigma, \pi_1\times \pi_2)=1\;,
\]
if and only if,
\[
\left\{ \begin{array}{l} m(\sigma_\otimes, \mathbf{r}_{\alpha_\sigma}(\pi_1\times\pi_2)) =1 \\ w \mbox{ avoids }321 \end{array}\right.\;.
\]

\end{conjecture}

\begin{remark}
The condition $x\leq w$ appearing in our initial definition of the set $\mathcal{I}(\pi_1,\pi_2)$ follows implicitly from Conjecture \ref{main-conj} by Proposition \ref{bruhat-prop}.
\end{remark}

Let us briefly discuss why the determination of the value of $m(\sigma_\otimes, \mathbf{r}_{\alpha_\sigma}(\pi_1\times\pi_2))$ is a task of low complexity, for any given case.

First, by Proposition \ref{mult-one} this value equals either $0$ or $1$. The proof of Proposition \ref{mult-one} also portraits a method of computation. We write $\sigma_\otimes = L(\Delta_1 + \Delta_2)\otimes \sigma'_\otimes$, where $\sigma'\in \irr$ is defined by a smaller multisegment. Then, by Proposition \ref{prop-count-mul} the question becomes whether there exists $s\in \mathcal{J}(\pi_1,\pi_2)$, for which both $m(L(\Delta_1 + \Delta_2), \tau^s)$ and $m(\sigma'_\otimes,  \mathbf{r}_{\alpha_{\sigma'}}(\delta^s))$ do not vanish. The collection of representations $\{\tau^s, \delta^s\}_{s\in \mathcal{J}(\pi_1,\pi_2)}$ is described in \cite{LapidKret}. In particular, $\tau^s, \delta^s$ are all products of two ladder representations. Hence, the value of $m(L(\Delta_1 + \Delta_2), \tau^s)$ is deduced immediately from Lemma \ref{mul-lem}, while $m(\sigma'_\otimes,  \mathbf{r}_{\alpha_{\sigma'}}(\delta^s))$ can be computed inductively.

\begin{proposition}\label{prop-conj}
The following conjectures are equivalent:
% in the sense that validity of each one of them would imply validity of all :

\begin{enumerate}

\item Conjecture \ref{main-conj}.

\item (Erez Lapid) Let $\pi_1, \pi_2\in\irr_{\langle \rho \rangle}$ be ladder representations with notation as above. Let us assume that $m_1-m_2 \in \{0,1\}$, the inequalities
\[
\left\{\begin{array}{l}\lambda_1^1 < \lambda_1^2 < \lambda_2^1 < \lambda_2^2 < \ldots < \lambda_{m_2}^2\;(< \lambda_{m_2+1}^1 ) \\

\mu_1^1 < \mu_1^2 < \mu_2^1 < \mu_2^2 < \ldots < \mu_{m_2}^2\;(< \mu_{m_2+1}^1 ) \end{array}\right.
\]
hold, and that $S(\lambda,\mu) = S_{m_1+m_2}$.

Then,
\[
\mathcal{I}(\pi_1,\pi_2)= \{w\in S(\lambda, \mu)\;:\; w\mbox{ avoids }321\}\;.
\]

\item Let $\pi_1, \pi_2\in\irr_{\langle \rho \rangle}$ be ladder representations with notation as above. Let us assume that $m_1-m_2 \in \{0,1\}$, and that
\[
\begin{array}{ll}
\lambda^1 = (0,2,\ldots, 2(m_1-1)),& \lambda^2= (1,3, \ldots, 2m_2-1), \\
\mu^1 = (2m_2, 2m_2 +2,\ldots, 2(m_2+m_1-1)), & \mu^2 = (2m_2+1, 2m_2+3,\ldots, 4m_2-1).
\end{array}
\]
Then,
\[
\mathcal{I}(\pi_1,\pi_2)= \{w\in S_{m_1+m_2}\;:\; w\mbox{ avoids }321\}\;.
\]

\item Let $\pi_1, \pi_2\in\irr_{\langle \rho \rangle}$ be ladder representations with notation as above. Let us assume that $m_1-m_2 \in \{0,1\}$, the inequalities
\[
\lambda_1^1 \leq \lambda_1^2 \leq \lambda_2^1 \leq \lambda_2^2 \leq \ldots \leq \lambda_{m_2}^2\;(\leq \lambda_{m_2+1}^1 )
\]
hold, and that $\gotM^x_{\lambda,\mu} = \gotM^e_{\lambda,\mu}$ (for the identity element $e\in S_{m_1+m_2}$).

Then,
\[
\mathcal{I}(\pi_1,\pi_2)= \{w\in S(\lambda, \mu)\;:\; w\mbox{ avoids }321\}\;.
\]

\item For all $\check{\lambda},\check{\mu}\in \mathcal{P}_m$ and all $w,w'\in S(\check{\lambda},\check{\mu})$ which avoid $321$ and for which $w\neq w'$, we have
\[
m(\sigma_\otimes, \mathbf{r}_{\alpha_\sigma}(\sigma'))=0\;,
\]
where $\sigma = L(\gotM^w_{\check{\lambda},\check{\mu}})$ and $\sigma' = L(\gotM^{w'}_{\check{\lambda},\check{\mu}})$.

\item For every integer $n$ and every $z\in S_n$ which avoids $321$, we have the following identity on values of Kazhdan-Lusztig polynomials:
\[
\sum_{\substack{ w_1\in S_{\lfloor n/2 \rfloor}, w_2\in S_{\lceil n/2 \rceil} \\ \widetilde{w_1\ast w_2} \leq z  }} \epsilon (w_1)\epsilon (w_2) P_{\widetilde{w_1\ast w_2}, z}(1) = 1\;.
\]
Here $\epsilon$ denotes the sign of a permutation and $\widetilde{w_1\ast w_2}$ denotes the operation described in Section \ref{sect-kl} after taking the partition of $\{1,\ldots,m\}$ into odd and even numbers.

\end{enumerate}

\end{proposition}
\begin{proof}
\underline{(1) implies (2):}

With the assumptions of (2), let $w\in S_{m_1+m_2}$ be a permutation which avoids $321$. Write $\sigma =L(\gotM^w_{\lambda,\mu})$. By Conjecture \ref{main-conj}, it is enough to show that $m(\sigma_\otimes, \mathbf{r}_{\alpha_\sigma}(\pi_1\times\pi_2))>0$. We will prove it by induction on $m_1+m_2$.

Let us write $\lambda= (\lambda_i)_{i=1}^{m_1+m_2}$ and $\mu= (\mu_i)_{i=1}^{m_1+m_2}$. Since $\lambda_1 = \lambda_1^1 < \lambda_1^2 = \lambda_2$, we can write $\sigma_\otimes = L([\lambda_1, \mu_{w(1)}]) \otimes \sigma'_\otimes$, where $\sigma' = L(\gotM^w_{\lambda,\mu} - [\lambda_1, \mu_{w(1)}])$.

By the formula of \cite{LapidKret}, for $t=1,2$, there exist Jacquet modules $\mathbf{r}_{\alpha_t}(\pi_t)$ for maximal parabolic subgroups which contain as irreducible factors the representations $\pi'_t\otimes \pi''_t$, respectively, where
\[
\pi'_1 =L\left( [\lambda_1^1, \mu_1^1]+ \sum_{2\leq i\,:\, \mu_i^1\leq \mu_{w(1)} } [\mu^2_{i-1}+1, \mu^1_i]  \right) , \quad \pi''_1 = L\left( \sum_{2\leq i\,:\, \mu_i^1\leq \mu_{w(1)} } [\lambda_i^1, \mu^2_{i-1}] + \sum_{i\,:\, \mu_i^1> \mu_{w(1)}}  [\lambda_i^1, \mu^1_i] \right)
\;,
\]
\[
\pi'_2 = L\left( \sum_{ i\,:\, \mu_i^2\leq \mu_{w(1)} } [\mu^1_{i}+1, \mu^2_i]  \right) , \quad \pi''_2 = L\left( \sum_{i\,:\, \mu_i^2\leq \mu_{w(1)} } [\lambda_i^2, \mu^1_i] + \sum_{i\,:\, \mu_i^2> \mu_{w(1)}}  [\lambda_i^2, \mu^2_i] \right)
\;.
\]
Note, that $\pi'_1$ and $\pi'_2$ are generic irreducible representations. In particular, $[\pi'_1\times \pi'_2]$ is a product of a sequence of segment representations in $\mathcal{R}$. More precisely, $[\pi'_1\times \pi'_2] = \prod_{i=1}^{w(1)}[L(\Delta_i)]$, with $b(\Delta_i) = e(\Delta_{i-1})+1$, for all $2\leq i\leq w(1)$, $b(\Delta_1) = \lambda_1^1 = \lambda_1$ and $e(\Delta_{w(1)}) = \mu_{w(1)}$. Hence, $m(L([\lambda_1,\mu_{w(1)}]), \pi'_1\times \pi'_2)>0$.

We leave the reader the verification that one of the ordered pairs of ladder representations, $(\pi''_1, \pi''_2)$ or $(\pi''_2, \pi''_1)$, satisfies again the assumptions of (2). In particular, by the induction assumption we have $m(\sigma'_\otimes, \mathbf{r}_{\alpha_{\sigma'}}(\pi''_1\times\pi''_2))>0$.

Thus, by applying Proposition \ref{prop-count-mul} we see that $m(\sigma_\otimes, \mathbf{r}_{\alpha_\sigma}(\pi_1\times\pi_2))>0$ holds as desired.
\\ \\
\underline{(2) implies (3):}
The conditions of (3) are a special case of the conditions of (2).
\\ \\
\underline{(3) implies (6):}

Let $\pi_1, \pi_2$ be the ladder representations with the assumptions of (3). Recall the ``determinantal" identity for ladder representations, which states that
\[
[\pi_t] = \sum_{w\in S_{m_t}} \epsilon(w) [M(\gotM^w_{\lambda_t,\mu_t})]\;,
\]
for $t=1,2$ (See e.g. \cite{LM} or \cite{Hender}).

Now, for any $z\in S_{m_1+m_2}$ which avoids $321$, the desired identity follows from the formula \eqref{coeff-form} in the proof of Theorem \ref{thm-indp-len} and the assumption $m(L(\gotM^z_{\lambda,\mu}), \pi_1\times\pi_2)=1$, that is given by (3).
\\ \\
\underline{(6) implies (4):}
For two permutations $w_t\in S(\lambda^t,\mu^t)$, $t=1,2$, recall the operations $\widetilde{w_1\ast w_2}\in Q(\lambda,\mu)$ and $w_1\ast w_2 \in S(\lambda,\mu)$ defined in Section \ref{sect-kl}. Note, that from the assumptions of (4), the partitions involved in these operations must divide $\{1,\ldots,m\}$ into odd and even numbers.

Since any $z\in S(\lambda,\mu)$ is the longest element in its corresponding double-coset of $S_{m_1+m_2}$, we have the known identity $P_{\widetilde{w_1\ast w_2},\,z} = P_{w_1\ast w_2,\, z}$ of Kazhdan-Lusztig polynomials (see \cite[Corollary 7.14]{humph-cox}).

Therefore, as before, formula \eqref{coeff-form} in the proof of Theorem \ref{thm-indp-len} (and the remark following it) implies that the identities of (6) amount to stating the equalities $m(L(\gotM^z_{\lambda,\mu}), \pi_1\times\pi_2)=1$, for all $z\in S(\lambda,\mu)$ which avoid $321$.
\\ \\
\underline{(4) implies (5):}
Let $\check{\lambda}= (\check{\lambda}_i),\check{\mu}=(\check{\mu}_i)\in \mathcal{P}_m$ be given. Suppose also that $w,w'\in S(\check{\lambda},\check{\mu})$ are given which avoid $321$. The fact that such permutations exist imply that there are no equalities of the form $\check{\lambda}_i = \check{\lambda}_{i+1} = \check{\lambda}_{i+2}$ or $\check{\mu}_i = \check{\mu}_{i+1} = \check{\mu}_{i+2}$, for any $i$.

Thus, we clearly can define ladder representations $\pi_1,\pi_2$ which satisfy all assumptions of (3), and for which $\check{\lambda} = \lambda$ and $\check{\mu}=\mu$.
 Write $\sigma,\sigma'$ as in (5).  From Proposition \ref{mult-one}, we know that $m(\sigma_\otimes, \mathbf{r}_{\alpha_\sigma}(\pi_1\times \pi_2))\leq 1$. By exactness of the Jacquet functor, we have
\[
1\geq m(\sigma_\otimes, \mathbf{r}_{\alpha_\sigma}(\pi_1\times \pi_2)) = \sum_{ z\in \mathcal{I}(\pi_1,\pi_2)} m(\sigma_\otimes, \mathbf{r}_{\alpha_\sigma}(L(\gotM^z_{\lambda,\mu})))\;.
\]
The statement of (5) now follows from (4) and the fact that $m(\sigma_\otimes, \mathbf{r}_{\alpha_\sigma}(\sigma))=1$.
\\ \\
\underline{(5) implies (1):}
One direction of Conjecture \ref{main-conj} follows immediately from Corollary \ref{cor-patt}, Proposition \ref{mult-one} and inequality \eqref{maj}.

Conversely, suppose that $m(\sigma_\otimes, \mathbf{r}_{\alpha_\sigma}(\pi_1\times\pi_2)) =1$, for $\sigma = L(\gotM^w_{\lambda,\mu})$ with $w\in S(\lambda,\mu)$ which avoids $321$. We assume the notations of Conjecture \ref{main-conj} here.

From exactness of $\mathbf{r}_{\alpha_\sigma}$, there is $\sigma'\in \irr$ with $m(\sigma',\pi_1\times \pi_2)>0$ and $m(\sigma_\otimes, \mathbf{r}_{\alpha_\sigma}(\sigma'))=1$. From Corollary \ref{cor-patt} we know that $\sigma' = L(\gotM^{w'}_{\lambda,\mu})$ for some $w'\in S(\lambda,\mu)$ which avoids $321$. Finally, by (5) $w = w'$ and $\sigma \cong \sigma'$.

\end{proof}

Let us finish by mentioning again that the validity of the identities appearing in Proposition \ref{prop-conj}(5) were verified by a computer program for many low rank cases (all cases of $n\leq 11$). In fact, such computations have initially led Erez Lapid to formulate a version of the conjectures of this section which he conveyed to the author.

\section{Smooth multisegments}\label{sect-smooth}
We would like to prove Conjecture \ref{main-conj} for a certain family of special cases. These involve representations $\sigma\in \irr_{\langle \rho \rangle}$ which are given as $\sigma = L(\gotM^w_{\lambda,\mu})$ for a permutation $w\in S(\lambda,\mu)$, which satisfies $m(L(\gotM^w_{\lambda,\mu}),M(\gotM^e_{\lambda,\mu}))=1$. By formula \eqref{kl-pol} in Section \ref{sect-kl}, this condition amounts to the Kazhdan-Lusztig polynomial identity $P_{e,w}\equiv 1$. In geometric terms, the condition is that the Schubert variety associated to the permutation $w$ is smooth.

In \cite{lak-san}, a combinatorial criterion for detecting smooth permutations was developed. It was shown that a permutation $w\in S_n$ satisfies $P_{e,w}\equiv 1$, if and only if, it avoids both $3412$ and $4231$. It clearly follows that among the $321$-avoiding permutations, those which satisfy the smooth condition are precisely those which avoid $3412$ as well. In fact, our interest in this section is solely on the combinatorial condition of $321$ and $3412$ avoidance.

Curiously, this collection of permutations was studied in \cite{fan-smooth}. Fan identified $321$ and $3412$ avoiding permutations as those which are given by a product of distinct simple transpositions. Their cardinality in the group $S_n$ was shown to be given by $F_{2n-1}$, where $F_k$ stands for the $k$-th Fibonacci number. Let us mention in that context the well known fact that the number of all $321$ avoiding permutations in $S_n$ (represents here the number of all cases of Conjecture \ref{main-conj} needed to be verified) is given by the $n$-th Catalan number.

\begin{lemma}\label{lem-simp-comb}
Suppose that $w\in S_n$ is a permutation that avoids both $321$ and $3412$. If $k=w(1)$, then for all $2\leq i\leq k-1$, $w(i) = i-1$.
\end{lemma}
\begin{proof}
Note, that because of the $321$ pattern avoidance we cannot have $w(j)<w(i)< k=w(1)$ for any $i<j$. %Hence, for all $j_1<j_2<k$, $w^{-1}(j_1)<w^{-1}(j_2)$.

Now suppose that $w(i_1)>k$ for a certain index $1< i_1\leq k-1$. Since the interval $(1,i_1)$ contains at most $k-3$ integers, there must be two indices $i_1< i_2<i_3$ for which $w(i_2),w(i_3) < k$. In particular, as seen above, $w(i_2)<w(i_3)$. The sequence $1<i_1<i_2< i_3$ then supplies a $3412$ pattern.

Thus, $w(i)< k$ for all $2\leq i\leq k-1$. In combination with the first observation of the proof, the statement follows.

\end{proof}

%\begin{lemma}
%Let $\Delta_1\prec \Delta_2\prec\ldots\prec\Delta_n$ be given segments in $\seg_{\langle \rho \rangle}$. Let $\gotN\in \mathbb{N}(\seg_{\langle \rho \rangle})$ be a multisegment such that $\Delta_n\prec \Delta$ for all $\Delta\in \underline{\gotN}$ and write $\sigma =  L\left( \sum_{i=1}^n \Delta_i + \gotN\right)\in \irr$. Suppose that
%\[
%m\left( \sigma_\otimes, \mathbf{r}_{\alpha_\sigma}\left(\lambda\left(\gotM^w_{\lambda,\mu}\right)\right)\right)>0\;,
%\]
%for some $w\in S(\lambda,\mu)$, which is a minimal presentation for $\gotM^w_{\lambda,\mu}$.

%Then, $w(i)=i$, for all $1\leq i\leq n$.
%\end{lemma}

\begin{proposition}\label{prop-mulone}
Let $\lambda, \mu\in \mathcal{P}_m$ and $w\in S(\lambda,\mu)$ be given. Suppose that $w$ avoids both $321$ and $3412$. Then, for every $x\in S(\lambda, \mu)$,
\[
m(\pi_{\otimes}, \mathbf{r}_{\alpha_\pi}(M(\gotM^x_{\lambda,\mu}))\leq 1
\]
holds, where $\pi = L(\gotM^w_{\lambda,\mu})$.

\end{proposition}

\begin{proof}
We prove by induction on $m$.

%Let us assume that $m(\pi_{\otimes}, \mathbf{r}_{\alpha_\pi}(\lambda(\gotM^x_{\lambda,\mu}))$ is positive and prove that it equals $1$.

Suppose first that $\mu_{w(1)}= \lambda_1-1$. Since $\mu_{w(1)} = \mu_{x(i')}$ for some $i'$, we must have $\lambda_{i'}= \lambda_1$. Thus, $\gotM^w_{\lambda,\mu},\gotM^x_{\lambda,\mu}$ can also be described as $\gotM^{\check{w}}_{\check{\lambda},\check{\mu}},\gotM^{\check{x}}_{\check{\lambda},\check{\mu}}$, respectively, where $\check{\lambda} = (\lambda_2,\ldots,\lambda_m) = \lambda - (\lambda_{i'})$, $\check{\mu}=\mu-(\mu_{w(i')})$, and $w,x$ are naturally flattened into $\check{w},\check{x}\in S(\check{\lambda},\check{\mu})$. Note that $\check{w}$ still satisfies the pattern avoidance condition, which allows for an application of the induction hypothesis.

Otherwise, let us assume $\lambda_1\leq \mu_{w(1)}$, and write the non-trivial representation $\sigma= L(\sum_{i=1}^l [\lambda_1, \mu_{w(i)}])$, where $l=1$ if $\lambda_1< \lambda_2$ or $l=2$ if $\lambda_1= \lambda_2$. Since $w\in S(\lambda,\mu)$ avoids $321$, it is easy to see that $\lambda_l < \lambda_{l+1}$. Hence, we can write $(\pi)_\otimes = \sigma \otimes (\pi')_\otimes$, for
\[
\pi' = L( \gotM^{\tilde{w}}_{\tilde{\lambda}, \tilde{\mu}})\;,\quad \tilde{\lambda} = \lambda - (\lambda_i)_{i=1}^l,\; \tilde{\mu} = \mu - (\mu_{w(i)})_{i=1}^l\,,
\]
with $w$ being flattened into $\tilde{w}\in S(\tilde{\lambda},\tilde{\mu})$. Note again that $\tilde{w}$ still avoids both $321$ and $3412$.

To avoid confusion, for a permutation $y \in S(\tilde{\lambda},\tilde{\mu})\subset S_{m-l}$, we will write $\underline{y}$ for the natural bijection
\[
\underline{y}:\{l+1,\ldots,m\}\to \{1,\ldots, m\}\setminus\{w(i)\}_{i=1}^l
\]
it represents when recalling the construction of $(\tilde{\lambda},\tilde{\mu})$ out of $(\lambda,\mu)$. In these terms $\underline{\tilde{w}}$ is the restriction of $w$.

By Proposition \ref{prop-count-mul} the statement will follow once we have established the following three claims:
\begin{enumerate}

\item\label{cond-1} For all $s\in \mathcal{J} :=\mathcal{J}(L([\lambda_1,\mu_{x(1)}]),\ldots,L([\lambda_m,\mu_{x(m)}]))$,
\[
m(\sigma, \tau^s)\leq1\;.
\]
\item\label{cond-2} If $m(\sigma, \tau^s)>0$ for some $s\in \mathcal{J}$, then $m((\pi')_\otimes, \delta^s)\leq 1$.

\item\label{cond-3} There is at most one $s\in \mathcal{J}$ for which $m(\sigma, \tau^s)m((\pi')_\otimes, \delta^s)>0$.

\end{enumerate}
The representation $\sigma$ is generic, since it is defined by pairwise unlinked segments. So, claim \eqref{cond-1} follows from Proposition \ref{prop-gen}.

Now, assume that $s\in \mathcal{J}$ is such that $m(\sigma, \tau^s)>0$. By Lemma \ref{lem-tech}, $[\delta^s] = [M(\gotM^{\tilde{x}}_{\tilde{\lambda},\tilde{\mu}})]$ for a certain $\tilde{x}\in S(\tilde{\lambda},\tilde{\mu})$. Claim \eqref{cond-2} now follows from the induction hypothesis.

Let us write $k=w(1)$. From the pattern avoidance condition we know by Lemma \ref{lem-simp-comb} that $w(i)= i-1$ for all $2\leq i\leq k-1$. This also means $\tilde{w}(i) =i $ for all $1\leq i\leq k-1-l$.

In particular $w(1)>w(2)$.

Recall that Lemma \ref{lem-tech} states further that $\underline{\tilde{x}}(i) = x(i)$ for all $i$ such that $\mu_{x(i)}> \mu_k$.  It follows from $w(1) = k$ and definition of $S(\lambda,\mu)$ that either $\mu_k < \mu_{k+1}$ or $k=m$. Thus, we know that $\underline{\tilde{x}}(i) = x(i)$ for all $i$ such that $x(i) > k$.

It is clear that $s$ is determined by $\tilde{x}$. In order to establish claim \eqref{cond-3} we are left to show there is at most one $\tilde{x}$ satisfying the conditions imposed so far and for which $m((\pi')_\otimes,M(\gotM^{\tilde{x}}_{\tilde{\lambda},\tilde{\mu}}))$ is positive.

By Proposition \ref{bruhat-prop} such $\tilde{x}$ must satisfy $\tilde{w}\geq \tilde{x}$ in the Bruhat order of $S_{m-l}$. Therefore, $\tilde{x}(i) = i$ for all $1\leq i\leq k-1-l$. In other words, $\underline{\tilde{x}}(i) = i-1$ for all $l+1\leq i \leq k-1$. We see that the values of $\underline{\tilde{x}}^{-1}$ are determined on $\{l,\ldots, k-2\}$ and on $\{k+1,\ldots,m\}$ by the prescribed conditions. Since $\underline{\tilde{x}}$ is a bijection, the last value of $\underline{\tilde{x}}^{-1}$ (on $k-1$) is determined as well. Therefore, condition \eqref{cond-3} follows from uniqueness of $\tilde{x}$.

\end{proof}

\begin{remark}
Incidentally, one can use Proposition \ref{prop-mulone} to reprove directly part of the Lakshmibai-Sandhya criterion for smoothness, in the following sense. We can pick $\lambda,\mu \in \mathcal{P}_m$ for which $S(\lambda,\mu)= S_m$. Then, by Proposition \ref{prop-mulone} and inequality \eqref{maj}, for a given $w\in S_m$ which avoids $321$ and $3412$, $m(L(\gotM^w_{\lambda,\mu}),M(\gotM^e_{\lambda,\mu}))\leq 1$ will hold. As explained at the head of this section this must imply that the associated Schubert variety to $w$ is smooth.
\end{remark}

\begin{theorem}\label{thm-smooth}
Suppose that $\pi_1=L(\gotM_1),\ldots, \pi_k=L(\gotM_k)$ are representations in $\irr_{\langle \rho \rangle}$. Suppose that $\gotM_1+ \ldots +\gotM_k = \gotM^x_{\lambda,\mu}$ for some $\lambda,\mu\in \mathcal{P}_m$ and $x\in S(\lambda,\mu)$.

Then, for all $\sigma= L(\gotM^w_{\lambda,\mu})$ with $w\in S(\lambda,\mu)$ which avoids both $321$ and $3412$,
\[
m(\sigma, \pi_1\times\cdots \times\pi_k)= m(\sigma_\otimes, \mathbf{r}_{\alpha_\sigma}(\pi_1\times\cdots\times\pi_k))\in\{0,1\}\;.
\]

\end{theorem}
\begin{proof}
Suppose that $\sigma= L(\gotM^w_{\lambda,\mu})$ with $w\in S(\lambda,\mu)$ which avoids both $321$ and $3412$.

As in the proof of Proposition \ref{gen-decomp}, $\pi_1\times\cdots\times\pi_k$ is a quotient of $M(\gotM_1)\times\cdots\times M(\gotM_k)$. Recall that $[M(\gotM_1)\times\cdots\times M(\gotM_k)] = [M(\gotM^x_{\lambda,\mu})]$, which means
\[
m(\sigma_\otimes, \mathbf{r}_{\alpha_\sigma}(\pi_1\times\cdots\times\pi_k)) \leq  m(\sigma_\otimes, \mathbf{r}_{\alpha_\sigma}(M(\gotM^x_{\lambda,\mu})))\leq 1\;
\]
by Proposition \ref{prop-mulone}.

Suppose now that $m(\sigma_\otimes, \mathbf{r}_{\alpha_\sigma}(\pi_1\times\cdots\times\pi_k))=1$. By exactness of the Jacquet functor, $m(\sigma_\otimes, \mathbf{r}_{\alpha_\sigma}(\theta))=1$ for a certain $\theta \in \irr$ such that $m(\theta, \pi_1\times\cdots\times\pi_k)>0$. If $\theta\cong \sigma$ we are done.

Assume that $\theta\not\cong \sigma$. As before, we know that $ m(\sigma_\otimes, \mathbf{r}_{\alpha_\sigma}(M(\gotM^x_{\lambda,\mu})))>0$. Hence, by Proposition \ref{bruhat-prop} we have $x\leq w$ and $\sigma$ is a subquotient of $M(\gotM^x_{\lambda,\mu})$. Since $\theta$ is also a subquotients of $M(\gotM^x_{\lambda,\mu})$, we get a contradiction from
\[
2\leq  m(\sigma_\otimes, \mathbf{r}_{\alpha_\sigma}(\theta)) + m(\sigma_\otimes, \mathbf{r}_{\alpha_\sigma}(\sigma))\leq m(\sigma_\otimes, \mathbf{r}_{\alpha_\sigma}(M(\gotM^x_{\lambda,\mu})))\;.
\]

\end{proof}

\begin{corollary}\label{cor-smooth}
Let $\pi_1=L(\gotM^e_{\lambda^1,\mu^1}), \pi_2=L(\gotM^e_{\lambda^2,\mu^2})\in \irr_{\langle \rho \rangle}$ be ladder representations, for some $\lambda^i = (\lambda^i_1< \ldots < \lambda^i_{m_i}), \mu^i = (\mu^i_1 <\ldots < \mu^i_{m_i}) \in \mathcal{P}_{m_i}$, $i=1,2$. Write $\lambda = \lambda^1 + \lambda^2$ and $\mu=\mu^1+\mu^2$.

%Let $\pi_1=L(\gotM_1), \pi_2=L(\gotM_2)$ be ladder representations in $\irr_{\langle \rho \rangle}$. Suppose that $\gotM^e_{\lambda^1,\mu^1}+ \gotM^e_{\lambda^2,\mu^2} = \gotM^x_{\lambda,\mu}$ for some $x\in S(\lambda,\mu)$.

Then, for all $\sigma= L(\gotM^w_{\lambda,\mu})$ with $w\in S(\lambda,\mu)$ which avoids $3412$, we have that
\[
m(\sigma, \pi_1\times \pi_2)=1\;,
\]
if and only if,
\[
\left\{ \begin{array}{l} m(\sigma_\otimes, \mathbf{r}_{\alpha_\sigma}(\pi_1\times\pi_2)) =1 \\ w \mbox{ avoids }321 \end{array}\right.\;.
\]

\end{corollary}
\begin{proof}
Follows from Theorem \ref{thm-smooth}, Corollary \ref{cor-patt} and Corollary \ref{cor-main-2}.
\end{proof}

\bibliographystyle{alpha}
\bibliography{propo2}{}

\end{document}